\documentclass{m2an}

\usepackage{amssymb,amsfonts,latexsym,amsmath,wrapfig,graphicx}

\usepackage{bm}

\usepackage{enumitem}

\usepackage{chngcntr}
\usepackage{xcolor}
\usepackage{braket}

\usepackage{hyperref}
\hypersetup{
  colorlinks   = true, 
  urlcolor     = blue, 
  linkcolor    = blue, 
  citecolor   = red 
}

\usepackage{mathrsfs}
\usepackage[capitalize]{cleveref}
\crefname{equation}{}{}

\numberwithin{equation}{section}

\theoremstyle{plain} 
\newtheorem{theorem}{Theorem}[section] 
\newtheorem{lemma}[theorem]{Lemma} 
\newtheorem{corollary}[theorem]{Corollary} 
\newtheorem{proposition}[theorem]{Proposition}
\theoremstyle{definition} 
\newtheorem{definition}[theorem]{Definition} 
\newtheorem{remark}[theorem]{Remark}
\newtheorem{example}[theorem]{Example}

\DeclareMathOperator{\rk}{rk}
\DeclareMathOperator{\Span}{Span}

\DeclareMathOperator{\Ran}{ran}

\DeclareMathOperator{\sgn}{sgn}

\newcommand{\RR}{\mathbb{R}}
\newcommand{\CC}{\mathbb{C}}
\newcommand{\NN}{\mathbb{N}}

\newcommand{\PP}{\mathbb{P}}
\newcommand{\VV}{\mathbb{V}}

\newcommand{\HC}{\mathfrak{H}}
\newcommand{\LC}{\mathfrak{L}}

\newcommand{\ol}{\overline}
\newcommand{\ul}{\underline}

\renewcommand{\vec}[1]{\mathbf{#1}}
\newcommand{\mi}[1]{\bm{#1}}

\newcommand{\dua}[2]{\langle #1, #2 \rangle}

\renewcommand{\epsilon}{\varepsilon}
\renewcommand{\kappa}{\varkappa}
\renewcommand{\phi}{\varphi}

\renewcommand{\setminus}{\smallsetminus}

\newcommand{\wt}[1]{\widetilde{#1}}
\newcommand{\wh}[1]{\widehat{#1}}

\newcommand{\iv}[1]{\,\mathrm{d}#1}
\newcommand{\grad}{\bm{\nabla}}
\newcommand{\lapl}{\triangle}

\newcommand{\ham}{\mathcal{H}}

\newcommand{\VC}{\mathfrak{V}}
\newcommand{\VF}{\mathfrak{v}}

\newcommand{\ene}{\mathcal{E}}

\newcommand{\bve}[1]{\overrightarrow{#1}}
\newcommand{\HA}{\mathrm{H}}

\newcommand{\ci}{\mathrm{CI}}
\newcommand{\cc}{\mathrm{CC}}
\newcommand{\occu}{\mathrm{occ}}
\newcommand{\virtu}{\mathrm{virt}}
\newcommand{\eff}{\mathrm{eff}}
\newcommand{\full}{\mathrm{full}}

\newcommand{\occ}[2]{\underline{#2}_{#1}}
\newcommand{\virt}[2]{\overline{#2}^{#1}}

\newcommand{\vertiii}[1]{{\left\vert\kern-0.25ex\left\vert\kern-0.25ex\left\vert #1
   \right\vert\kern-0.25ex\right\vert\kern-0.25ex\right\vert}}

\DeclareSymbolFont{bbold}{U}{bbold}{m}{n}
\DeclareSymbolFontAlphabet{\mathbbold}{bbold}

\begin{document}

\title{Coupled-Cluster Theory Revisited\\\vspace*{2em} Part I: Discretization}

\author{Mih\'aly A. Csirik}\address{Hylleraas Centre for Quantum Molecular Sciences, Department of Chemistry, University of Oslo, P.O. Box 1033 Blindern, N-0315 Oslo, Norway (\email{m.a.csirik@kjemi.uio.no})} 
\author{Andre Laestadius}\address{Department of Computer Science, Oslo Metropolitan University, P.O. Box 4 St. Olavs plass, NO-0130 Oslo, Norway (\email{andre.laestadius@oslomet.no})}\sameaddress{1}

\date{Received June 22, 2022. Accepted November 14, 2022.}


\begin{abstract}
In a series of two articles, we propose a comprehensive mathematical framework for Coupled-Cluster-type methods. These methods aim at accurately solving the many-body Schr\"odinger equation.
In this first part, we rigorously describe the discretization schemes involved in Coupled-Cluster methods using graph-based concepts. This allows us to discuss different methods in a unified and more transparent manner, including multireference methods. Moreover, we derive the single-reference and the Jeziorski--Monkhorst multireference 
Coupled-Cluster equations in a unified and rigorous manner.
\end{abstract}



\maketitle

\section{Introduction}

The Coupled-Cluster (CC) method 
is one of the most popular methods in computational quantum chemistry among Hartree--Fock (HF)
 and Density-Functional Theory (DFT). 
In its full generality, the quantum many-body problem is intractable, and it is one of the greatest challenges of quantum mechanics 
to devise practically useful methods to approximate the solutions of the many-body Schr\"odinger equation. 
Although the stationary Schr\"odinger equation itself is a linear eigenvalue problem, it is extremely high-dimensional
even for a few particles and a low-dimensional one-particle space.\footnote{The dimension is ${K \choose N}$, where $N$ is the number of particles, and $K$ is dimension of
the one-particle Hilbert space.} Here, we focus on those fermionic systems which are described by the
so-called \emph{molecular Hamilton operator}---on which most electronic-structure models are based in quantum chemistry.
The Galerkin method applied to the Schr\"odinger equation (sometimes combined with an initial HF ``guess'') is branded Configuration Interaction (CI) in computational quantum chemistry; unfortunately, its applicability is limited due to the aforementioned high-dimensionality issue. The HF method is perhaps conceptually the simplest, whereby the ground state is approximated by minimizing the energy of the system 
over Slater determinants; the resulting Euler--Lagrange equations constitute a nonlinear eigenvalue problem that yields the HF ground state.
HF theory has attracted much interest in the mathematical physics community, see e.g. \cite{lieb1974solutions,lions1987solutions,bach1992error,bach1997there,cances2000convergence,cances2003computational,solovej2003ionization,friesecke2003multiconfiguration,lewin2018existence}. 
The spiritual successor to the statistical mechanics-motivated Thomas--Fermi theory---DFT---is the single most used method in quantum chemistry, and some of its 
mathematical aspects are also highly non-trivial \cite{lieb1983density,eschrig1996fundamentals,lewin2018semi,lewin2020local}. 

CC theory is a vast and highly active subfield of quantum chemistry, consisting of many variants and refinements. 
However, among the aforementioned methods, the CC approach has arguably received the least attention in the mathematics community.

\subsection{Previous work}
It is beyond the scope of this paper to give a historical review of the CC method and its vast number of variants. The interested reader is pointed to
\cite{helgaker2014molecular,kummel1978many,bishop1991overview,bartlett2007coupled,shavitt2009many}. The survey article \cite{wilson2013methods}
is somewhat more mathematically-oriented and also proposes a rather general framework.

The analysis of the single-reference CC method by R. Schneider and T. Rohwedder \cite{schneider2009analysis,rohwedder2013continuous,rohwedder2013error} serves as a starting point of our description of
the CC discretization. In Part II, we will briefly summarize \cite{schneider2009analysis} and its follow-up works, where the context is more appropriate.

\subsection{Outline}

It is our intention to present both known and new results in a self-contained manner and primarily with a mathematical audience in mind.
In \cref{secbg}, we describe the setting of the quantum-mechanical problems the CC theory is aimed at. 
 Next, \cref{seccicc} gives a rough sketch of the most basic CI and CC methods.

We begin our discussion in \cref{secexord} with the definition of a partial order relation which will be used to encode the relevant transitions
of the system, called \emph{excitations}. This partial order, and the induced lattice operations will be used in \cref{exgrsec} to
define the \emph{excitation graph}, which fully describes the CC discretization scheme. We give a few examples of the generality of our concepts
and also extend the definition of the excitation graph to the multireference (MR) case.
  After this, the corresponding \emph{excitation operators} (\cref{exop}), 
\emph{cluster operators} (\cref{clustop}) and \emph{cluster amplitude spaces} (\cref{secamp}) are constructed, which are the essential building blocks for the formulation
of any CC-like method.

In \cref{ccdersec}, we give short derivations of the conventional single-reference (SRCC) and Jeziorski--Monkhorst multireference (JM-MRCC) methods. We do so by generalizing the known procedure
which is based on perturbation theory.

In \cref{appgraph} we calculate various graph-theoretic properties of the excitation graph.
In \cref{optimalmulti} we propose a method based on linear programming to select reference determinants for the multi-reference setting in an optimal way.

\section{Background}\label{secbg}
In this section we collect the concepts and results that are necessary for the forthcoming discussion.
For proofs and more about the mathematical foundations of quantum mechanics,
see e.g. \cite{reed1978i,reed1975ii,reed1978iv,gustafson2011mathematical,yserentant2010regularity,hislop2012introduction,lieb2010stability}.

 The spectrum of a linear operator $A$ is written $\sigma(A)$,
the elements of its discrete spectrum as $\ene_n(A)$, where $n=0,1,2,\ldots$, if $A$ is bounded from below.
We use the usual notation $[A,B]=AB-BA$ for the commutator. The transpose of $A$ is denoted as $A^\dag$. For normed spaces $V$ and $W$, the symbol $\mathcal{L}(V,W)$ denotes normed
space of \emph{bounded} linear mappings $V\to W$ endowed with the operator norm $\|\cdot\|_{\mathcal{L}(V,W)}$. Furthermore, $V^*$ denotes the (continuous) dual space. 

\subsection{Function spaces}\label{secfnsp}

In the context of many-body quantum mechanics, the Lebesgue-, and Sobolev spaces $L^2(\RR^3)$ and $H^1(\RR^3)$ are viewed as ``one-particle spaces''.
We ignore spin for simplicity and consider $L^2(\RR^3)$ and $H^1(\RR^3)$ as \emph{real} Hilbert spaces, again for clarity. These choices are justified for our
model Hamiltonian (see \cref{molhamsec} below), but we remark that all the forthcoming considerations can be trivially extended to the more general setting.  The one-particle spaces
 are then used to define the
$N$-particle \emph{fermionic} spaces (see e.g. \cite{lieb1996analysis})
$$
\LC^2=\bigwedge^N L^2(\RR^{3}), \quad \text{and} \quad \HC^1=\LC^2 \cap H^1(\RR^{3N}),
$$
endowed with the inner products
$$
\dua{\Psi}{\Phi}= \int_{\RR^{3N}} \Psi(\vec{X}) \Phi(\vec{X}) \iv{\vec{X}}
$$
and
$$
\dua{\Psi}{\Phi}_{\HC^1}=\dua{\Psi}{\Phi} + \sum_{k=1}^N \int_{\RR^{3N}}  \grad_{\vec{x}_k} \Psi(\vec{X}) \cdot \grad_{\vec{x}_k} \Phi(\vec{X}) \iv{\vec{X}},
$$
respectively. Here, $\vec{X}=(\vec{x}_1,\ldots,\vec{x}_N)\in\RR^{3N}$
and $\vec{z}\cdot\vec{w}$ denotes the Euclidean inner product. Also,
$\grad_{\vec{x}_k}=(\partial_{x_k^1},\partial_{x_k^2},\partial_{x_k^3})$ is the distributional gradient operator acting on the $k$th triple of the arguments.
The norms corresponding to $\dua{\cdot}{\cdot}$ and $\dua{\cdot}{\cdot}_{\HC^1}$ are denoted as $\|\cdot\|$ and $\|\cdot\|_{\HC^1}$, respectively.
We define the second order Sobolev space as $\HC^2=\LC^2 \cap H^2(\RR^{3N})$.

Let $K\geq N$ or $K = \infty$ and assume that an $L^2$-orthonormal \emph{(spin-)orbital set} $\mathcal{B}_K=\{\phi_p\}_{p=1}^K\subset H^1(\RR^3)$ is given.
 We define the subspace $H_K^1(\RR^3)=\Span\mathcal{B}_K\subset H^1(\RR^3)$. 
Corresponding to $\mathcal{B}_K$ we can construct the set of \emph{Slater determinants}
$$
\mathfrak{B}_K=\{ \Phi_{\alpha} \in \HC^1 : 1\le\alpha_1<\ldots<\alpha_N\le K,\; \Phi_{\alpha}(\mathbf{X})=N!^{-1/2}\det(\phi_{\alpha_i}(\vec{x}_j))_{1\le i,j\le N} \}.
$$
Then $\mathfrak{B}_K$ is $\LC^2$-orthonormal. Set
$$
\HC^1_K=\Span\mathfrak{B}_K\subset\HC^1. 
$$

The negative exponent Sobolev space $\HC^{-1}$ will also be used in the sequel, which is given by the continuous dual space $(\HC^1)^*$.
We will exploit that the dense continuous embeddings $\HC^1\subset\LC^2\subset\HC^{-1}$ hold true (see e.g. \cite{adams2003sobolev}), i.e. the spaces in question form a \emph{Gelfand triple}.

\begin{remark}\label{gelfandrmrk}
An important remark is in order. Recall that $V\subset H\subset V^*$ are said to form a Gelfand triple 
if $V$ is a real reflexive Banach space, $H$ is real separable Hilbert space and the embedding $V\subset H$ is continuous
and $V$ is dense in $H$ (see e.g. \cite[Section 23.4]{zeidler1985nonlinear}). It follows that for any $\Psi\in H$
there is a $\wh{\Psi}\in V^*$ such that $\dua{\wh{\Psi}}{\Phi}_{V^*\times V}=\dua{\Psi}{\Phi}_H$ for all $\Phi\in V$,
and the mapping $H\to V^*$ given by $\Psi\mapsto\wh{\Psi}$ is linear, injective and continuous. Hence, the embedding $H\subset V^*$ is also continuous (and dense). Henceforth, we will write $\Psi$ in place of $\wh{\Psi}$ for brevity.
\end{remark}

\noindent\textbf{Convention:} We will drop the subscript from $\dua{\cdot}{\cdot}$, as it is \emph{either} obvious from its arguments that the duality pairing $\dua{\cdot}{\cdot}_{V^*\times V}$ has
to be used, \emph{or} both $\dua{\cdot}{\cdot}_{V^*\times V}$ and $\dua{\cdot}{\cdot}_H$ are acceptable due to the Gelfand triple setting as
discussed above. In particular, we apply this convention to $\HC^1\subset\LC^2\subset\HC^{-1}$, to the cluster amplitude spaces discussed in \cref{secamp}
and also in the general framework of \cref{ccdersec}.

\subsection{Schr\"odinger Hamiltonian}\label{molhamsec}

In this section, we introduce the model Hamiltonian for concreteness.
Let $V,w:\RR^3\to\RR$ be Kato class potentials: $V,w\in L^{3/2}(\RR^3)+L^\infty_\epsilon(\RR^3)$\footnote{
By definition $f\in L^{3/2}(\RR^3)+L^\infty_\epsilon(\RR^3)$, if for every $\epsilon>0$ there is an $f_1\in L^{3/2}(\RR^3)$ and $f_2\in L^\infty(\RR^3)$ with $\|f_2\|_\infty<\epsilon$
so that $f=f_1+f_2$.} with $w$ even and define the quadratic form $\ene$ on $\HC^1$ as
$$
\ene(\Psi)=\frac{1}{2}\|\grad \Psi\|^2  + \int_{\RR^{3N}} \Bigg(\sum_{1\le i\le N} V(\vec{x}_i) + \sum_{1\le i<j\le N} w(\vec{x}_i-\vec{x}_j)\Bigg)
|\Psi(\vec{X})|^2 \, \mathrm{d}\vec{X}
$$
for any $\Psi\in\HC^1$. For every $\epsilon>0$, there is a $C_\epsilon>0$ so that Kato's inequality (see e.g. \cite{friesecke2003multiconfiguration} for a detailed proof),
$$
\frac{1-\epsilon}{2} \|\grad \Psi\|^2 - C_\epsilon \|\Psi\|^2\le \ene(\Psi) \le \frac{1+\epsilon}{2} \|\grad \Psi\|^2 + C_\epsilon \|\Psi\|^2
\quad\text{for all}\quad \Psi\in \HC^1,
$$
holds true. This implies that the quadratic form induced by $V$ and $w$ is infinitesimally form bounded with respect to $-\lapl$. 
The KLMN theorem (see e.g. \cite{reed1975ii}) implies that there exists a unique self-adjoint operator $\ham:D(\ham)\to\LC^2$ associated to $\ene$, having form domain $Q(\ham)=Q(\ene)=\HC^1$
and being lower semibounded. This $\ham$ is given by
$$
(\ham\Psi)(\vec{X})=\sum_{1\le i\le N} \Bigg[-\frac{1}{2}\lapl_{\vec{x}_i} + V(\vec{x}_i)\Bigg]\Psi(\vec{X}) +  \sum_{1\le i<j\le N} w(\vec{x}_i-\vec{x}_j)\Psi(\vec{X}),
$$
for all $\Psi\in D(\ham)$ and $\vec{X}\in\RR^{3N}$.
Kato's inequality implies that there is a constant $M>0$, such that
\begin{equation}\label{Hbound}
\dua{\ham\Psi}{\Phi}\le M\|\Psi\|_{\HC^1}\|\Phi\|_{\HC^1}
\end{equation}
for all $\Psi,\Phi\in\HC^1$.
Thus, $\ham$ can be extended to a bounded mapping $\HC^1\to\HC^{-1}$, which we denote with the same symbol.
We say that $\Psi\in\HC^1$ and $\ene\in\RR$ satisfy the \emph{weak Schr\"odinger equation} if  $\dua{\ham\Psi}{\Phi}=\ene\dua{\Psi}{\Phi}$ for all $\Phi\in\HC^1$.

As far as the finite-dimensional case $K<\infty$ is concerned, we simply consider the Galerkin projection of the weak Schr\"odinger equation.
More precisely, let $\HC^1_K\subset\HC^1$ be as defined in \cref{secfnsp}. Then $\Psi\in\HC^1_K$ and $\ene\in\RR$ are said to satisfy the \emph{projected Schr\"odinger equation}
if $\dua{\ham\Psi}{\Phi}=\ene\dua{\Psi}{\Phi}$ for all $\Phi\in\HC_K^1$.

The so-called (electronic) molecular Hamiltonian $\ham$ corresponds to the special case
$$
V(\vec{x})=-\sum_{1\le j\le M} \frac{Z_j}{|\vec{x}-\vec{r}_j|}\quad\text{and}\quad w(\vec{x})=\frac{1}{|\vec{x}|},
$$
where $Z_j\in\NN$ ($j=1,\ldots,M$) and $\vec{r}_1,\ldots,\vec{r}_M\in\RR^{3}$ denote the charges and the positions of the $M\in\NN$ nuclei.

\subsection{The CI and the CC method}\label{seccicc}

We now give a very rough description of the single-reference CI and CC methods. For the rigorous derivations, we refer to \cref{ccdersec}.

In a preliminary step---typically using the Hartree--Fock method---the \emph{reference determinant}
$$
\Phi_0=N!^{-1/2}\det(\phi_{i}(\vec{x}_j))_{1\le i,j\le N}
$$
is constructed and normalized so that $\|\Phi_0\|=1$. We restrict our discussion here to the case when relevant the function spaces are real. 
The \emph{occupied orbitals} $\mathcal{B}_{\occu}=\{\phi_p\}_{p=1}^N\subset H^1(\RR^3)$ are 
extended to a basis $\mathcal{B}_K=\{\phi_p\}_{p=1}^K\subset H^1_K(\RR^3)$ by adding $K-N$ \emph{virtual orbitals} $\mathcal{B}_{K,\virtu}=\{\phi_p\}_{p=N+1}^K$, so that
$\mathcal{B}_K=\mathcal{B}_{\occu}\cup \mathcal{B}_{K,\virtu}$. Here, $K=\infty$ is allowed. The orthonormal set $\mathcal{B}_K$ generates the Slater determinants $\mathfrak{B}_K$
and the subspace $\HC^1_K\subset\HC^1$ (see \cref{secfnsp}).
For later convenience, we introduce the space $\HC^{1,\perp}$ as the $\LC^2$-orthogonal complement of $\Span\{\Phi_0\}$ in $\HC^1$. Further, we also 
set $\HC^{1,\perp}_K=\HC^{1,\perp}\cap\HC^1_K$. Further, we define $\LC^{2,\perp}$ as the $\LC^2$-orthogonal complement of $\Span\{\Phi_0\}$ in $\LC^2$.

In both the CI and the CC method, the Schr\"odinger equation $\ham\Psi=\ene\Psi$ is solved based on the reference wavefunction $\Phi_0$.
For simplicity,\footnote{Although the CI method is more general.}  we consider the case when
$\Psi$ is sought after in the form $\Psi=\Phi_0+\ul{\Psi}$, where $\dua{\ul{\Psi}}{\Phi_0}=0$. In other words, $\Psi$ is calculated via a \emph{correction} $\ul{\Psi}$
to $\Phi_0$.
Note that $\dua{\Psi}{\Phi_0}=1$, which is called the
\emph{intermediate normalization} condition. If the ``targeted'' wavefunction $\Psi$ happens to be orthogonal to the reference determinant $\Phi_0$,
then the Ansatz $\Psi=\Phi_0+\ul{\Psi}$ cannot yield a solution.

The \emph{Full Configuration Interaction} (FCI) method
can be summarized as follows: find $\ul{\Psi}\in\HC^{1,\perp}_K$ such that
\begin{equation}\label{FCIi}
\dua{\ham(\Phi_0+\ul{\Psi})}{\Phi}=\ene_{\ci}\dua{\Phi_0+\ul{\Psi}}{\Phi}\quad\text{ for all $\Phi\in\HC^{1,\perp}_K$.}
\end{equation}
Here, $\ene_{\ci}=\|\Psi\|^{-2}\dua{\ham\Psi}{\Psi}$ 
is called the \emph{CI-}, or \emph{variational energy}.
 The \emph{projected CI} method is simply the Galerkin projection of the previous problem to some finite dimensional subspace
$\VC_d\subset\HC^{1,\perp}_K$, i.e. to find $\ul{\Psi}_d\in\VC_d$ such that
\begin{equation}\label{CIi}
\dua{\ham(\Phi_0+\ul{\Psi}_d)}{\Phi_d}=\ene_{d,\ci}\dua{\Phi_0+\ul{\Psi}_d}{\Phi_d}\quad\text{ for all $\Phi_d\in\VC_d$.}
\end{equation}
 The choice of the Galerkin subspace $\VC_d$ is typically based on so-called truncation rank, for instance $\VC_d=\VC_{\textrm{SD}}$, is
the span of singly-, and doubly excited determinants in $\mathfrak{B}_K$. The corresponding (projected) CI method in this case is designated as ``CISD''.

The CI equations are more commonly expressed using \emph{cluster operators}. A cluster operator $C:\LC^2\to\LC^2$ is a bounded linear operator
that is a linear combination of special products of
fermionic creation and annihilation operators $a_i^\dag$ and $a_i$ (see Part II for a definition), so that the action of each such product is to replace some occupied orbitals 
$\mathcal{B}_{\occu}$ with the same number of virtual orbitals $\mathcal{B}_{K,\virtu}$ in a Slater determinant (see \cref{2qex}). A cluster operator $C$ can therefore be parametrized with 
the said linear-combination coefficients, which are denoted by the lower case $c$ and are called \emph{cluster amplitudes}. The vector space of all cluster amplitudes
is denoted by $\VV$.  
There is a one-to-one correspondence between functions in $\LC^{2,\perp}$ (resp. $\HC^{1,\perp}$) and functions of the form $C\Phi_0$, where $C:\LC^2\to\LC^2$ 
(resp. $C:\HC^1\to\HC^1$) is a cluster operator (see \cite{rohwedder2013continuous}).
Therefore, \cref{FCIi} can be expressed as follows: find a cluster operator $C$ (or, equivalently cluster amplitudes $c$), such that 
\begin{equation}\label{FCIC}
\dua{\ham(I+C)\Phi_0}{S\Phi_0}=\ene_{\ci}(c)\dua{(I+C)\Phi_0}{S\Phi_0}\quad\text{ for all cluster operators $S$.}
\end{equation}
Here, $\ene_{\ci}(c)=\|(I+C)\Phi_0\|^{-2}\dua{\ham(I+C)\Phi_0}{(I+C)\Phi_0}$.
Although this might seem an unnecessary complication at first, cluster operators are essential for the formulation of the CC method.

In the CC method, the ``exponential Ansatz'' is assumed for the intermediately normalized wavefunction $\Psi$. Substituting $\Psi=e^T\Phi_0$ into the Schr\"odinger
equation, where $T$ is a cluster operator, we get 
\begin{equation}\label{eTschr}
\ham e^T\Phi_0=\ene_{\cc}e^T\Phi_0,
\end{equation}
for some $\ene_{\cc}\in\RR$.
First, to determine $\ene_{\cc}$ we premultiply \cref{eTschr} by $e^{-T}$ ($e^T$ is always invertible), and take the inner product with $\Phi_0$ to
obtain the \emph{CC energy}
\begin{equation}\label{ccene}
\ene_{\cc}:=\ene_{\cc}(t)=\dua{e^{-T}\ham e^T\Phi_0}{\Phi_0},
\end{equation}
where we used the normalization $\|\Phi_0\|=1$. Second, by premultiplying \cref{eTschr} by $e^{-T}$ again, but now testing against functions in $\HC^{1,\perp}_K$, we get
the \emph{Full CC} (FCC) method: find cluster amplitudes $t_*\in\VV$ such that 
\begin{equation}\label{FCCi}
\dua{e^{-T_*}\ham e^{T_*}\Phi_0}{S\Phi_0}=0,\quad\text{ for all $s\in\VV$.}
\end{equation}
The \emph{projected CC} method is the Galerkin projection of the FCC problem with respect to some subspace $\VV_d\subset\VV$.
More precisely, the task is to find $t_{d,*}\in\VV_d$ such that
\begin{equation}\label{TCCi}
\dua{e^{-T_{d,*}}\ham e^{T_{d,*}}\Phi_0}{S_d\Phi_0}=0\quad\text{ for all $s_d\in\VV_d$.}
\end{equation}
For the moment, we denote the corresponding CC energy by $\ene_{d,\cc}$.
Again, $\VV_d$ is based on some truncation, such as SD, in which case the corresponding method is called ``CCSD''.

We now discuss the relation between CI and CC. 
It was shown that the FCI \cref{FCIi} and the FCC \cref{FCCi} equations are equivalent, see \cite[Theorem 5.3]{rohwedder2013continuous}.
\begin{theorem}[Equivalence of FCI and FCC]\label{fccfci}
The problems \cref{FCIC} and \cref{FCCi} are equivalent, i.e. a full CC solution $t_*$ is also full CI solution $I+C_*=e^{T_*}$, and vice versa. Moreover, 
 $\ene_\cc(t_*)=\ene_\ci(c_*)$ holds true.
\end{theorem}

However, the corresponding Galerkin-projected problems are \emph{not} equivalent.
Further, while $\ene_{\ci}\le\ene_{d,\ci}$ due to the Rayleigh--Ritz variational principle, the same is not true for the CC method and numerical experience undoubtedly shows that \emph{there is no obvious 
relation} in general between $\ene_\cc=\ene_{\ci}$ and $\ene_{d,\cc}$; this last phenomenon is called the \emph{nonvariational property}
of CC theory. Note that according to \cref{fccfci}, FCC \emph{is} variational.

Despite this, the gains of CC over CI are significant. First, by construction, the CC method is size-consistent, even when truncated \cite[Theorem 4.10]{schneider2009analysis}.
This property is crucial for the precise determination of various chemical properties.
Second, the evaluation of expressions involving the \emph{similarity-transformed Hamilton operator} $e^{-T}\ham e^T$ is greatly eased by the formula 
\begin{equation}\label{bchham}
e^{-T}\ham e^T=\sum_{j=0}^4 \frac{1}{j!} [\ham,T]_{(j)},
\end{equation}
see \cite[Theorem A.1]{schneider2009analysis},\footnote{Their proof is straightforward to adapt to the more general Hamilton operator defined in \cref{molhamsec}.} where the iterated commutators are given by $[\ham,T]_{(0)}=\ham$ and $[\ham,T]_{(j)}=[[\ham,T]_{(j-1)},T]$ for $j\ge 1$.
Equation~\cref{bchham} may be referred to as the terminating Baker--Campbell--Hausdorff series, and it makes the computer implementation of CC methods feasible even for moderately sized systems. 
In particular, the Slater--Condon rules imply that the CC energy can be computed as\footnote{Actually, the term $\dua{\ham T_1\Phi_0}{\Phi_0}$ vanishes if $\Phi_0$ is the Hartree--Fock solution (Brillouin theorem).}
\begin{equation}\label{cceneprol}
\ene_{\cc}(t)=\dua{\ham(I+T_1+T_2+\tfrac{1}{2}T_1^2)\Phi_0}{\Phi_0}.
\end{equation}
Furthermore, \cref{bchham} also implies that the polynomial system \cref{FCCi} (and hence its Galerkin projection \cref{TCCi}) is quartic in terms of
the cluster amplitudes $t$.
Despite their apparent simplicity, the CC equations usually involve many complicated terms and even their assembly is a nontrivial task.
In summary, the CC method approximates an extremely high-dimensional linear problem \cref{FCIi} by a low-dimensional nonlinear problem \cref{TCCi}.

\section{Coupled-Cluster discretization}\label{secccdisc}

Using an appropriate string of creation and annihilation operators, any fermionic state can be changed to any other one (see Part II, or \cite{helgaker2014molecular,solovej2007many}).
In our context, a set of $N$ occupied orbitals is given; its complement is called the set of virtual orbitals.
The action of an excitation operator on a Slater determinant consists of annihilating a few occupied orbitals and creating the same number of virtual orbitals (hence the
particle number $N$ is conserved). A de-excitation operator amounts to the reverse action: annihilating some virtual orbitals and creating the same
number of occupied ones. Obviously, any $N$-particle Slater determinant can be obtained by acting with an appropriate excitation operator on the ``reference state'', which 
is the $N$-particle Slater determinant of all the occupied orbitals. However, it might also be
possible to arrive at the same Slater determinant from another state through successive excitations. The concrete relationships are nontrivial and this section is
devoted to their description.

\subsection{Excitation order}\label{secexord}

Let $\Lambda$ be a countable set called the \emph{orbital set} and let $2^\Lambda$ denote the power set of $\Lambda$. In concrete examples, we will often use the 
numbers $\Lambda=\{1,2,3,\ldots\}$ to label the elements of $\Lambda$ for the sake of simplicity, and set $K=|\Lambda|$. 
Let $N\ge 1$ denote the number of particles,
and set $S=\{\alpha\in 2^\Lambda : |\alpha|=N\}$, the elements of which are called ($N$\emph{-particle}) \emph{states}.
The particle number $N$ is assumed to be fixed throughout.
Fix $M\ge 1$ \emph{reference states} 
$$
\Omega=\{0_1,\ldots,0_M\}\subset S.
$$
For every $m=1,\ldots,M$ define 
$$
L_m=S\setminus(\Omega\setminus\{0_m\})
$$
and on it, the partial order relation
$$
\alpha\preceq_m \beta \quad\Longleftrightarrow\quad \occ{m}{\beta}\subset\occ{m}{\alpha}\quad\text{and}\quad \virt{m}{\alpha}\subset\virt{m}{\beta}
$$
for any $\alpha,\beta\in L_m$, where 
$$
\occ{m}{\alpha}=\alpha\cap 0_m\quad\text{and}\quad
\virt{m}{\alpha}=\alpha\cap (0_m)^c,
$$
and the complement is to be understood relative to $\Lambda$.
 According to commonly used nomenclature, we call $\occ{m}{\alpha}$ the \emph{occupied part of $\alpha$ w.r.t. $0_m$}
and $\virt{m}{\alpha}$ the \emph{virtual part of $\alpha$ w.r.t. $0_m$.} This partial order relation is a generalization of \cite[Definition 4.2]{rohwedder2013continuous}.
By definition, $L_m=\{\alpha\in S : 0_m\preceq_m\alpha\}$ and 
for the sake of convenience, we introduce the notations $\ol{S}=S\setminus\Omega$ and $\ol{L}_m=L_m\setminus\{0_m\}$. 
Note that the reference states are defined \emph{not} to be comparable with respect to $\preceq_m$ with each other.

The partial order $\preceq_m$ generates the \emph{join} and \emph{meet} lattice operations
\begin{align*}
	\alpha\vee_m\beta&=\big(\occ{m}{\alpha}\cap\occ{m}{\beta}\big)\cup\big(\virt{m}{\alpha}\cup\virt{m}{\beta}\big),\\
	\alpha\wedge_m\beta&=\big(\occ{m}{\alpha}\cup\occ{m}{\beta}\big)\cup\big(\virt{m}{\alpha}\cap\virt{m}{\beta}\big),
\end{align*}
for all $\alpha,\beta\in L_m$. 
Furthermore, we introduce the orthocomplementation $\alpha^\perp=\Lambda\setminus\alpha$.

For the so-called \emph{single-reference} (SR) case, $M=1$ and we will make the convention that all the $m$ indices are dropped from the notation. 
For the next result, we extend $\preceq$, $\vee$ and $\wedge$ to the whole $2^\Lambda$.
\begin{proposition}\label{boolean}
The structure $B=(2^\Lambda,\vee,\wedge,0,1,\perp)$ is a Boolean algebra, that is, a distributive, bounded lattice in which the de Morgan laws hold true.
Here, we set $1:=\Lambda$, the identity for $\wedge$.
\end{proposition}

A similar statement holds true in the \emph{multi-reference} (MR) case, for the individual structures $B_m=(2^\Lambda,\vee_m,\wedge_m,0_m,1,\perp)$.
Even though the algebraic structure on $B$ is nice, the subset $S$ loses this structure.
In fact, $S$ is \emph{not} a sublattice of $B$, since for example $\alpha\vee_m\beta,\alpha\wedge_m\beta\not\in S$ for distinct $\alpha$ and $\beta$
with $\ul{\alpha}=\ul{\beta}=\emptyset$. The reason why we stated \cref{boolean}, however, is because we will exploit the operational rules for $\vee$, $\wedge$ and $\perp$
on a few occasions; for instance, in the following trivial result.
\begin{lemma}\label{abinv}
Let $\gamma,\beta\in 2^\Lambda$ be such that $\beta\preceq_m\gamma$. Then, $\alpha\vee_m\beta=\gamma$ if and only if $\alpha=\beta^\perp\wedge_m\gamma$.
\end{lemma}
\begin{proof}
We have
$$
\alpha\vee_m\beta=(\gamma\wedge_m\beta^{\perp})\vee_m\beta=(\gamma\vee_m\beta)\wedge_m(\beta^{\perp}\vee_m\beta)=(\gamma\vee_m\beta)\wedge_m 1=\gamma\vee_m\beta=\gamma,
$$ 
where in the last step we used $\beta\preceq_m\gamma$. Further, if $\alpha'\vee_m\beta=\gamma$ as well, then $\alpha'\vee_m\beta=\alpha\vee_m\beta$. By joining $\beta^{\perp}$ to
both sides, we get $\alpha'=\alpha$.
\end{proof}

The poset $(L_m,\preceq_m)$ also admits a rank function which makes it a \emph{graded poset}. Being a graded poset means that there is a \emph{rank function} $\rk_m:L_m\to\NN$
satisfies $\rk_m(\alpha)<\rk_m(\beta)$ whenever $\alpha\prec_m\beta$, and $\rk_m(\beta)=\rk_m(\alpha)+1$ if
there is no element $\gamma$ such that $\alpha\prec_m\gamma\prec_m\beta$. The choice $\rk_m(\alpha)=|\virt{m}{\alpha}|$ is easily seen to satisfy the
requirements. Obviously, the maximum value that $\rk_m(\alpha)$ can take is $N$. 
 For a geometric description of the rank function, see \cref{optimalmulti}.

\begin{figure}[ht]\label{mrgraph}
	\centering
\includegraphics[width=0.8\linewidth]{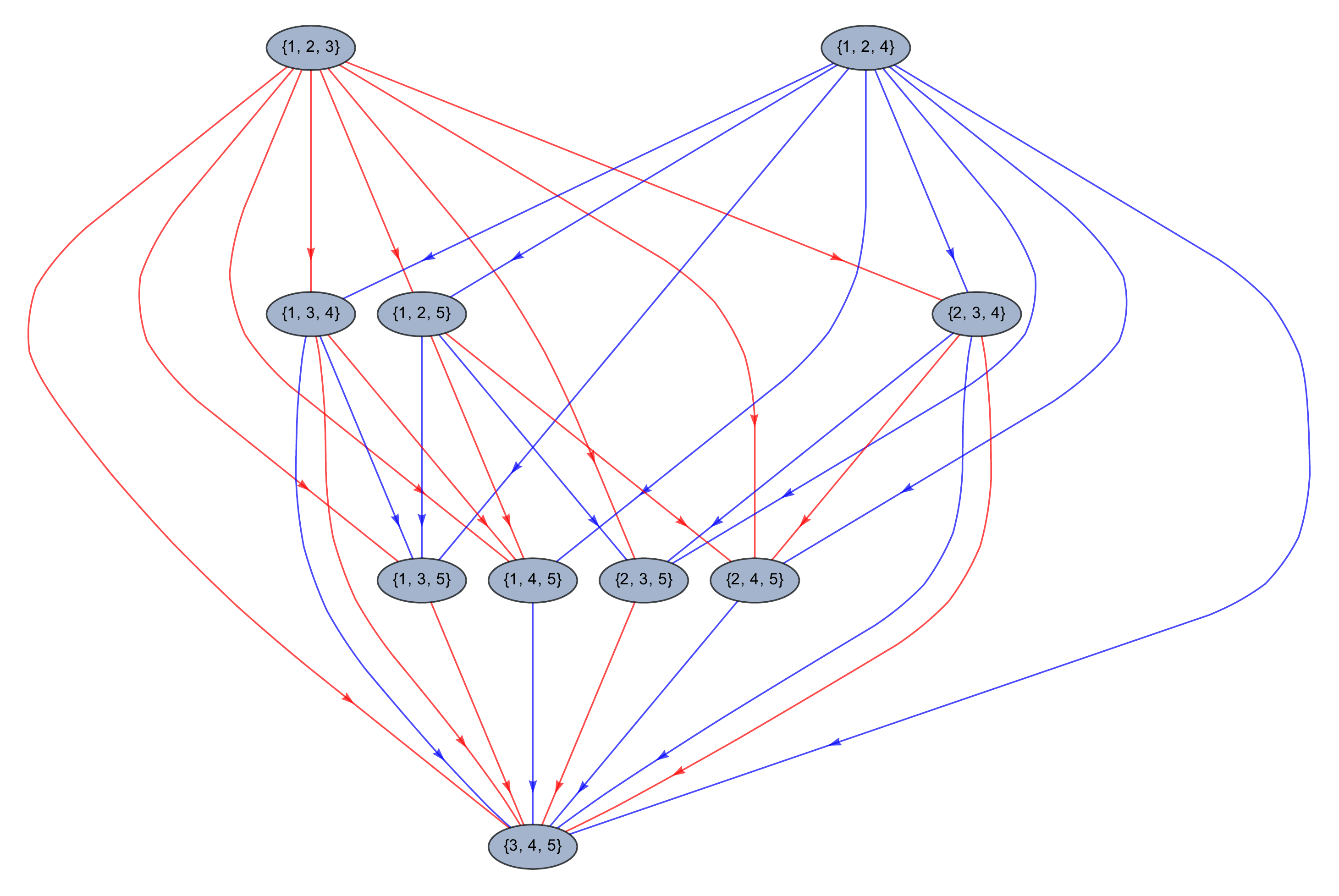}
\caption{Full multi-reference excitation multigraph for $\Lambda=\{1,\ldots,5\}$ and $0_1=\{1,2,3\}$, $0_2=\{1,2,4\}$. 
The edges corresponding to $0_1$ and $0_2$ are shown in red and blue, respectively.}
\end{figure}

\subsection{Excitation graphs}\label{exgrsec}

As we remarked in the previous section, $L_m$ fails to be a sublattice of the Boolean algebra $B_m$. 
Therefore, let us consider pairs $(\alpha,\beta)\in L_m\times L_m$ for which $\alpha\vee_m\beta\in L_m$. In other words, pairs $(\alpha,\beta)\in L_m\times L_m$ for which
$|\occ{m}{\alpha}\cap\occ{m}{\beta}|+|\virt{m}{\alpha}\cup\virt{m}{\beta}|=N$,
or, using the inclusion-exclusion principle $|A\cup B|=|A|+|B|-|A\cap B|$, we can equivalently write
\begin{equation}\label{incexl}
|\occ{m}{\alpha}\cup\occ{m}{\beta}|+|\virt{m}{\alpha}\cap\virt{m}{\beta}|=N,
\end{equation}
since $|\occ{m}{\alpha}|+|\virt{m}{\alpha}|=N$ and $|\occ{m}{\beta}|+|\virt{m}{\beta}|=N$ by hypothesis.
While still $\alpha\vee_m\beta\in L_m$ even in the case $\virt{m}{\alpha}\cap\virt{m}{\beta}\neq\emptyset$, we wish to avoid that possibility on  physical grounds.
Namely, such an operation would introduce a repeated virtual orbital
in the state $\alpha\vee_m\beta$ and that is not allowed on the account of the Pauli exclusion principle. We note in passing that 
$\alpha\vee_m\beta\in L_m$ is equivalent to $\alpha\wedge_m\beta\in L_m$.  
In conclusion, we restrict our attention to the set 
\begin{equation}\label{graphL}
\mathcal{L}_m=\{ (\alpha,\beta)\in L_m\times L_m: \;|\occ{m}{\alpha}\cup\occ{m}{\beta}|=N\; \
\text{and}\; |\virt{m}{\alpha}\cap\virt{m}{\beta}|=0 \}.
\end{equation}
Hence, if $(\alpha,\beta)\in\mathcal{L}_m$, then we have $\alpha\vee_m\beta\in L_m$. The set $\mathcal{L}_m$ is symmetric to the diagonal 
(which it does not contain), i.e. $(\alpha,\beta)\in\mathcal{L}_m$ iff $(\beta,\alpha)\in\mathcal{L}_m$ and $(\alpha,\alpha)\not\in\mathcal{L}_m$. Also, $(0_m,\alpha),(\alpha,0_m)\in\mathcal{L}_m$ for any $\alpha\in L_m$. 
Furthermore,
the rank function $\rk_m$ is additive on $\mathcal{L}_m$ in the sense that
$$
\rk_m(\alpha\vee_m\beta)=\rk_m(\alpha)+\rk_m(\beta),
$$
for any $(\alpha,\beta)\in \mathcal{L}_m$.  This property may also seen to be a reason why we want to exclude the case $\virt{m}{\alpha}\cap\virt{m}{\beta}\neq\emptyset$.
Indeed, it could also be taken as the \emph{definition} of $\mathcal{L}_m$.
\begin{proposition}\label{Lmreform}
The set $\mathcal{L}_m$ can be written as
$$
\mathcal{L}_m=\{(\alpha,\beta)\in L_m\times L_m : \alpha\vee_m\beta\in L_m \;\mathrm{and}\; \rk_m(\alpha\vee_m\beta)=\rk_m(\alpha)+\rk_m(\beta) \}.
$$
\end{proposition}
\begin{proof}
Let $\mathcal{L}_m'$ denote the set on the right hand side of the preceding equation. Then, it is clear from the above that $\mathcal{L}_m\subset\mathcal{L}_m'$. Conversely,
suppose that $(\alpha,\beta)\in\mathcal{L}_m'$. Then, $|\virt{m}{\alpha}\cup\virt{m}{\beta}|=|\virt{m}{\alpha}|+|\virt{m}{\beta}|$, from which $|\virt{m}{\alpha}\cap\virt{m}{\beta}|=0$
using the inclusion-exclusion principle.
Since $\alpha\vee_m\beta\in L_m$, \cref{incexl} holds true, and 
we have that $|\occ{m}{\alpha}\cup\occ{m}{\beta}|=N$, so $(\alpha,\beta)\in \mathcal{L}_m$. Hence, $\mathcal{L}_m\supset\mathcal{L}_m'$.
\end{proof}

The set $\mathcal{L}_m$ is used for our main definition.

\begin{definition}\label{exgraphdef}
The digraph $G_m^{\full}=(L_m,E_m^{\full})$ is called the \emph{full (SR) excitation graph w.r.t. $0_m$}, where
$$
E_m^{\full}=\{(\beta,\alpha\vee_m\beta)\in L_m\times L_m : (\alpha,\beta)\in \mathcal{L}_m, \;\alpha\neq 0_m \}.
$$
A subgraph $G_m=(L_m,E_m)$, $E_m\subset E_m^\full$ is said to be an \emph{(SR) excitation (sub)graph w.r.t. $0_m$}.
\end{definition}
Notice that we excluded $\alpha=0_m$ to omit loop edges. \cref{abinv} can be refined in the following manner.
\begin{lemma}\label{abinvgraph}
Let $(\beta,\gamma)\in E_m^{\full}$. Then $\alpha=\beta^\perp\wedge_m\gamma\in L_m$ is the unique $\alpha$ such that $\alpha\vee_m\beta=\gamma$.
\end{lemma}
\begin{proof}
Using \cref{abinv}, we can uniquely solve the equation $\alpha\vee_m\beta=\gamma$ for $\alpha$
to obtain $\alpha=\beta^\perp\wedge_m\gamma\in 2^\Lambda$. Therefore, $(\beta,\alpha\vee_m\beta)\in E_m^{\full}$, which implies that $\alpha\in L_m$
using the definition of $E_m^\full$.
\end{proof}
\begin{corollary}\label{noparallel}
The digraph $G_m^{\full}$ does not contain parallel edges.
\end{corollary}

Various graph-theoretic quantities of the single-reference excitation graph are calculated in \cref{appgraph}.

A digraph $G=(V,E)$ is said to be \emph{transitive} if $(u,v)\in E$ and $(v,w)\in E$ imply $(u,w)\in E$.
It follows by induction that, if $G$ is transitive, and whenever $G$ contains a directed path $((v_0,v_1),(v_1,v_2),\ldots,(v_{n-1},v_n))$, then
$(v_0,v_n)\in E$.

\begin{proposition}\label{gfulltrans}
The digraph $G_m^\full$ is transitive.
\end{proposition}
\begin{proof}
Suppose that $(\gamma_0,\gamma_1)\in E^\full_m$ and $(\gamma_1,\gamma_2)\in E^\full_m$. Then there exists $\alpha$ and $\beta$ such that $\gamma_1=\alpha\vee_m\gamma_0$
and $\gamma_2=\beta\vee_m\gamma_1$. Since $\gamma_2=(\alpha\vee_m\beta)\vee\gamma_0$ by the associativity of $\vee_m$, it follows easily from \cref{Lmreform} that
$(\gamma_0,\alpha\vee_m\beta)\in \mathcal{L}_m$. Therefore, 
$$
(\gamma_0,\gamma_2)=(\gamma_0,(\alpha\vee_m\beta)\vee_m\gamma_0)\in E^\full_m,
$$
 which is what we wanted to show.
\end{proof}

Transitivity of certain subgraphs, and of $G_m^\full$ itself will come up later, since vaguely speaking this property will imply the algebraic
closedness of the set of excitation operators that we attach to the edges (see \cref{exop} and \cref{clustop}).

We label the edges of $G_m^{\full}$ with their corresponding $\alpha$. Thus, to every directed edge $(\beta,\alpha\vee_m\beta)\in E_m^{\full}$ there
corresponds a map $x_{m,\alpha}:L_m\to L_m$ defined with the instruction $x_{m,\alpha}(\beta)=\alpha\vee_m\beta$. This way, the digraph $G_m^{\full}$ may be interpreted as a commutative diagram 
(cf. \cref{exop}).
Note that a label $x_{m,\alpha}$ may appear on multiple edges. 

Furthermore, for any subgraph $G_m=(L_m,E_m)$, we introduce the
\emph{set of excitations $\Xi(G_m)\subset \ol{L}_m$ of $G_m$} via
\begin{equation}\label{xig}
\Xi(G_m)=\{\alpha\in \ol{L}_m : (\beta,\alpha\vee_m\beta)\in E_m \;\text{for some}\; \beta\in L_m\}.
\end{equation}
Note that the excitations are indexed with the same set $L_m$ as the states themselves, but in general $\Xi(G_m)\neq \ol{L}_m$.
Nonetheless, for the full excitation graph $G^\full_m$, we have in fact $\Xi(G_m^\full)=\ol{L}_m$.

The reason why explicitly stated that we are considering the ``full'' excitation graphs
is that, in practice, one is forced to ignore the ``degree of freedom'' (called ``cluster amplitudes'', see \cref{secamp}) corresponding to some edges.\footnote{
Note that the vertex set is still the ``full'' vertex set $L_m$---some vertices might become isolated. }
This is done by considering certain subsets of the full edge set $E_m^{\full}$.

\begin{definition}\label{exsub}
An excitation subgraph $G_m=(L_m,E_m)$ is said to be a \emph{consistent subgraph (of $G_m^\full$)} if $E_m\subset E_m^\full$, and
whenever $(\beta,\alpha\vee_m\beta)\in E_m$ for some $\beta\in L_m$ and $\alpha\in L_m$, then $(\beta',\alpha\vee_m\beta')\in E_m$ for all $\beta'\in L_m$.
\end{definition}

The consistency criterion can be rephrased as follows: for a fixed $\alpha\in\Xi(G_m)$, \emph{either} $E_m$ contains the whole ``orbit'' 
$\{(\beta,\alpha\vee_m\beta)\in E_m : \beta\in L_m\}$ \emph{or} it does not contain it at all. Note that the set $\Xi(G_m)$ can equally well be
used to define a consistent subgraph.

\begin{definition}\label{truncgraph}
For a given $r=1,\ldots,N$, define $G_m(r)=(L_m,E_m(r))$, where
$$
E_m(r)=\{(\beta,\alpha\vee_m\beta)\in E_m^\full : \beta\in L_m,\;\alpha\in\ol{L}_m\;\text{such that}\; \rk_m(\alpha)=r  \}.
$$
The subgraph $G_m(r_1,\ldots,r_\rho)=(L_m, E_m(r_1,\ldots,r_\rho))$ is called a \emph{rank-truncated excitation subgraph} if
$$
E_m(r_1,\ldots,r_\rho)=E_m(r_1)\cup\ldots\cup E_m(r_\rho)\quad\text{ for} \quad  r_1,\ldots,r_\rho\in\{1,\ldots,N\}.
$$
We refer to $G_m(1)$, $G_m(1,2)$, $G_m(1,2,3)$, etc. more colloquially as $G_m(\mathrm{S})$, $G_m(\mathrm{SD})$, $G_m(\mathrm{SDT})$, etc. 
\end{definition}

Rank-truncation does not introduce isolated vertices in $G_m(r_1,\ldots,r_\rho)$ as long as one of the $r_j$'s is 1. However, in the doubles (D) case,
$G(\textrm{D})$ \emph{does} in fact produce isolated vertices so that vertices of odd rank cannot be reached.
Also, note that these truncated subgraphs like $G_m(\mathrm{S})$ and $G_m(\mathrm{SD})$ are \emph{not} transitive in general.

We shall summarize these observations in the next theorem. Recall that a digraph is said to be \emph{weakly connected} if every pair of vertices has an undirected path between them.

\begin{theorem}\label{rkgraphthm}
Let $G_m=G_m(r_1,\ldots,r_\rho)$ be a rank-truncated excitation subgraph. Then the following is true.
\begin{enumerate}[label=\normalfont(\roman*)]
\item $G_m$ is a consistent subgraph.
\item $G_m$ is weakly connected if one of the $r_j$'s is 1.
\end{enumerate}
\end{theorem}
\begin{proof}
Obvious from the definition. 
\end{proof}

Next, we briefly consider two rather ``exotic'' CC-like methods to demonstrate the generality of the excitation graph concept.

\begin{example}\label{tccgraph}
The excitation graph corresponding to the Tailored CC method (see e.g. \cite{faulstich2019analysis}) can be described as follows. In this SR method ($M=1$), the orbital set
$\Lambda$ is partitioned
according to $\Lambda_{\mathrm{CAS}}=\{1,\ldots,N,N+1,\ldots,k\}$ and $\Lambda_{\mathrm{ext}}=\Lambda\setminus\Lambda_{\mathrm{CAS}}$ for some $k=N,\ldots,|\Lambda|$.
This induces a splitting $L=L(\mathrm{CAS})\dot{\cup} L(\mathrm{ext})$, where
$$
L(\mathrm{CAS})=\{\alpha\in L : \alpha\subset \Lambda_{\mathrm{CAS}}\}, \quad L(\mathrm{ext})=L\setminus L(\mathrm{CAS}).
$$
Furthermore, the edge set $E^{\full}$ may also be split accordingly
$$
E(\mathrm{CAS})=\{(\beta,\alpha\vee\beta)\in E^{\full} : \alpha\subset \Lambda_{\mathrm{CAS}}, \; \beta\in L \},\quad\text{and}\quad E(\mathrm{ext})=E^{\full}\setminus E(\mathrm{CAS}).
$$
In other words, $E(\mathrm{CAS})$ contains excitations which change CAS occupied orbitals to CAS virtual ones,
and as such, no edge in $E(\mathrm{CAS})$ leaves $L(\mathrm{CAS})$ that starts from $L(\mathrm{CAS})$.
It is easy to see that both $G(\mathrm{CAS})=(L,E(\mathrm{CAS}))$ and $G(\mathrm{ext})=(L,E(\mathrm{ext}))$ are transitive and consistent subgraphs.
\end{example}

\begin{example}\label{kowalskigraph}
A generalization of $E(\text{CAS})$ is the ``CAS-type subalgebra'' (denoted as ``$\mathfrak{g}^{(N)}(R,S)$'' in \cite{kowalski2018properties}), which is constructed from
two given subsets $\Lambda_R\subset \{1,\ldots,N\}$ and $\Lambda_S\subset\{N+1,\ldots\}$. Define $\Lambda_{\mathrm{int}}=\Lambda_R \dot{\cup} \Lambda_S$
and $\Lambda_{\mathrm{ext}}=\Lambda\setminus \Lambda_{\mathrm{int}}$. This induces a splitting $L=L(\mathrm{int})\dot{\cup}L(\mathrm{ext})$,
where 
$$
L(\mathrm{int})=\{\alpha\in L : \alpha\subset \Lambda_{\mathrm{int}} \}, \quad\text{and}\quad L(\mathrm{ext})=L\setminus L(\mathrm{int}).
$$
The edge set $E^{\full}$ decomposes as
$$
E(\text{int})=\{(\beta,\alpha\vee\beta)\in E^{\full} :  \alpha\subset \Lambda_{\mathrm{int}}, \; \beta\in L \},\quad\text{and}\quad E(\text{ext})=E^{\full}\setminus E(\text{int}).
$$
In other words, $E(\text{int})$ contains excitations that replace some orbitals in $\Lambda_R$ with ones in $\Lambda_S$.
Then $G(\mathrm{int})=(L,E(\mathrm{int}))$ and $G(\mathrm{ext})=(L,E(\mathrm{ext}))$ are transitive and consistent subgraphs. 
Clearly, \cref{tccgraph} can be recovered with the choice $\Lambda_R=\{1,\ldots,N\}$, $\Lambda_S=\{N+1,\ldots,k\}$.
\end{example}

Finally, we define excitation graph in the multireference case, which is a natural extension of the above concepts.

\begin{remark}\label{differentvac}
An important warning is in order. In general, $\alpha\vee_m\beta$ may or may not be equal to $\alpha\vee_\ell\beta$ for $m\neq\ell$. 
In fact, take $\Lambda=\{1,2,\ldots,7\}$ and $0_1=\{1,2,3\}$, $0_2=\{1,2,4\}$. Then, with $\alpha=\{1,3,5\}$ and $\beta=\{2,6,7\}$,
we have $\alpha\vee_1\beta=\alpha\vee_2\beta=\{5,6,7\}$. On the other hand, with $\alpha=\{2,3,4\}$ and $\beta=\{1,2,5\}$, we have $\alpha\vee_1\beta=\{2,4,5\}$,
but $\alpha\vee_2\beta=\{2,3,5\}$. Note that in the first case, we actually have $(\alpha,\alpha\vee_1\beta)\in E_1^{\full}$ and $(\alpha,\alpha\vee_2\beta)\in E_2^{\full}$,
i.e. a double edge.
\end{remark}

\begin{definition}\label{exmultidef}
The \emph{full MR excitation multigraph w.r.t. $\Omega$}, $G^{\full}=(L,E^{\full})$ is defined as the union of the individual full SR excitation graphs 
$G_m^{\full}=(L_m,E_m^{\full})$ for all $m=1,\ldots,M$, i.e.
$$
L=\bigcup_{m=1}^M L_m, \quad E^{\full}=\biguplus_{m=1}^M E_m^{\full},
$$
where $\uplus$ denotes multiset union.
\end{definition}

Note that as opposed to the SR graph $G_m^{\full}$, the MR graph $G^{\full}$ might have parallel edges (called 
``redundant'' excitations), this justifies that $G^{\full}$ was introduced as a multigraph. 
Notice that other references cannot be ``reached'' from a given one (see \cref{mrgraph}).
An algorithm for choosing the set of reference states $\Omega=\{0_m\}_{m=1}^M$ in an optimal way, adhering to some given criteria is described in \cref{optimalmulti}.

\subsection{Excitation operators}\label{exop}

Recall that $\Omega = \{0_m\}_{m=1}^M$ denotes the set of references, and that $L_m$ does \emph{not} contain the other reference states $\Omega\setminus\{0_m\}$.
The construction described below is to be repeated for every $m=1,\ldots,M$ separately.

First, we fix an ordering of the indices in $\alpha\in S$.
Then, for every element $\alpha=\{\alpha_1,\ldots,\alpha_N\}\in S$ we assign the lexicographically ordered $N$-tuple 
$$
\alpha^<=(\alpha_1^<, \ldots,\alpha_N^<)\in \Lambda^N, \quad \alpha_1^<<\ldots<\alpha_N^<, \quad\text{where}\quad \alpha_j^<\in\alpha.
$$
Without loss of generality, we can assume that the orbital indices contained in $0_m$ are strictly less than 
the virtual indicies $\Lambda\setminus 0_m$.

As in \cref{secfnsp}, fix an orthonormal set $\mathcal{B}_K=\{\phi_p\}_{p\in\Lambda}\subset H^1(\RR^3)$ and the corresponding Slater determinants 
\begin{equation}\label{detbas}
\mathfrak{B}_K=\{ \Phi_{\alpha} \in \HC^1 : \alpha\in S, \; \Phi_{\alpha}(\vec{X})=N!^{-1/2}\det(\phi_{\alpha_i^<}(\vec{x}_j))_{1\le i,j\le N} \}.
\end{equation}
Recall the notation $\HC^1_K\subset\HC^1$ for the subspace spanned by $\mathfrak{B}_K$; which is allowed to be finite-, or infinite-dimensional
depending on $K=|\Lambda|$.

\begin{definition}\label{exopmain}
Let $G_m=(L_m,E_m)$ be a subgraph of $G^\full_m$.
The family of linear operators $X_{\alpha}^{(m)}:=X_{\alpha}(G_m):\HC^1_K\to \HC^1_K$ given by 
$$
X_{\alpha}(G_m) \Phi_\beta=\begin{cases}
\sigma(\alpha,\beta)\Phi_{\alpha\vee_m\beta} & (\beta,\alpha\vee_m\beta)\in E_m \\
0 & (\beta,\alpha\vee_m\beta)\not\in E_m
\end{cases}
$$
for each $\alpha\in\Xi(G_m)$ and $\beta\in S$, and extended boundedly and linearly to the whole space $\HC^1_K$ (see \cite{rohwedder2013continuous}) is called the family of \emph{excitation operators on $G_m$}. 
Here, $\sigma(\alpha,\beta)$ is the sign of the permutation $\pi(\alpha,\beta)$ that puts the $N$-tuple
$((\virt{m}{\beta})^<,(\virt{m}{\alpha})^<)$
in lexicographical order.
\end{definition}

Assuming $\Xi(G_m)\neq\emptyset$, by the definition of $\Xi(G_m)$ (see \cref{xig}) for every $\alpha\in\Xi(G_m)$ 
there is some $\beta\in L_m$ such that $(\beta,\alpha\vee_m\beta)\in E_m$ and therefore $X_{\alpha}(G_m)\not\equiv 0$. 
Recalling $\rk_m(\alpha\vee_m\beta)=\rk_m(\alpha) + \rk_m(\beta)$ (see \cref{Lmreform}), we can roughly say that an excitation operator $X_\alpha(G_m)$ increases the rank by $\rk_m(\alpha)$.

Since $G_m=(L_m,E_m)$ is a subgraph of $G_m^\full=(L_m,E_m^\full)$, some excitations 
might be missing, i.e. $\Xi(G_m)\subset \Xi(G_m^\full)$. The next result shows that the excitation operators constructed for a \emph{consistent subgraph} $G_m$
(see \cref{exsub})
are precisely the same as the ones constructed for $G_m^\full$, with some of the excitation operators possibly missing. This explains the use of the word ``consistent''.

\begin{theorem}\label{consex}
Let $G_m$ be a \emph{consistent} subgraph of $G_m^\full$. Then, 
$$
X_\alpha(G_m)\equiv X_\alpha(G_m^\full)\quad\text{ for all $\alpha\in\Xi(G_m)$.}
$$
\end{theorem}
\begin{proof}
Fix $\alpha\in\Xi(G_m)$, then by \cref{xig} and \cref{exsub}, $(\beta,\alpha\vee_m\beta)\in E_m$ for all $\beta\in L_m$.
Consequently, $X_{\alpha}(G_m)\Phi_\beta=X_{\alpha}(G_m^\full)\Phi_\beta$ for all $\beta\in S$.
\end{proof}

Based on this result, if $G_m$ is consistent, it is safe to drop the ``$G_m$'' from the notation $X_{\alpha}(G_m)$ and simply denote the excitation operators by $X_{\alpha}^{(m)}$,
or by $X_\alpha$ in the SR case. However, it is important to note that for a given $\alpha$, $X_{\alpha}^{(m)}\neq X_{\alpha}^{(\ell)}$ in general for differing reference states $m\neq\ell$,
see \cref{differentvac}. 

The excitation operators enjoy nice algebraic properties which we summarize in the next theorem (cf. \cite[Lemma 2.5]{rohwedder2013continuous}). 
\begin{theorem}\label{exopthm}
Let $G_m=(L_m,E_m)$ be a consistent subgraph of $G^\full_m$ and let $\{X_{\alpha}^{(m)}\}_{\alpha\in\Xi(G_m)}$ denote the set of excitation operators on $G_m$.
Then the following properties hold true.
\begin{enumerate}[label=\normalfont(\roman*)]
\item (commutativity) For all $\alpha,\beta\in \Xi(G_m)$, there holds $X_{\alpha}^{(m)} X_{\beta}^{(m)}=X_{\beta}^{(m)} X_{\alpha}^{(m)}$. In detail, for any $\gamma\in S$,
$$
X_{\alpha}^{(m)} X_{\beta}^{(m)} \Phi_\gamma=\begin{cases}
\sigma(\alpha,\beta\vee_m\gamma)\sigma(\beta,\gamma) \Phi_{\alpha\vee_m\beta\vee_m\gamma} & (\beta\vee_m\gamma,\alpha\vee_m\beta\vee_m\gamma),\\
&(\gamma,\beta\vee_m\gamma)\in E_m \\
0 & \text{otherwise}
\end{cases}
$$
\item If $G_m$ is transitive, then $\{0\}\cup\{\pm X_\alpha^{(m)}\}_{\alpha\in\Xi(G_m)}$ is multiplicatively closed.
In particular, $\{0\}\cup\{\pm X_\alpha^{(m)}\}_{\alpha\in\Xi(G_m^\full)}$ is multiplicatively closed.
\item (nilpotency) For all $\alpha\in \Xi(G_m)$, $(X_{\alpha}^{(m)})^2=0$.
\end{enumerate}
\end{theorem}
\begin{proof}
To see (i), first observe that if $(\beta\vee_m\gamma,\alpha\vee_m(\beta\vee_m\gamma))$, $(\gamma,\beta\vee_m\gamma)\in E_m$, then
$(\alpha\vee_m\gamma,\beta\vee_m(\alpha\vee_m\gamma))$, $(\gamma,\alpha\vee_m\gamma)\in E_m$ due to the consistent subgraph property of $G_m$.
It is obvious that $\Phi_{\alpha\vee_m\beta\vee_m\gamma}=\Phi_{\beta\vee_m\alpha\vee_m\gamma}$ from the commutativity of $\vee_m$.
It remains to prove $\sigma(\alpha,\beta\vee_m\gamma)\sigma(\beta,\gamma)=\sigma(\beta,\alpha\vee_m\gamma)\sigma(\alpha,\gamma)$.
Let $\pi_1$, $\pi_2$ and $\tau_1$, $\tau_2$ be the permutations that put $((\ol{\beta}\cup \ol{\gamma})^<,\ol{\alpha}^<)$, $(\ol{\gamma}^<,\ol{\beta}^<)$
and $((\ol{\alpha}\cup \ol{\gamma})^<,\ol{\beta}^<)$, $(\ol{\gamma}^<,\ol{\alpha}^<)$, respectively, in lexicographic order.
Then $\pi_1\circ\pi_2=\tau_1\circ\tau_2=\sigma$, where $\sigma$ is the permutation that puts $(\ol{\alpha},\ol{\beta},\ol{\gamma})$ in lexicographic order.
The claim follows from the multiplicativity of the $\sgn$ function on permutations.

For (ii), suppose that $G_m$ is transitive and that $\alpha,\beta\in\Xi(G_m)$. Using (i), either $X_\alpha^{(m)}X_\beta^{(m)}=\pm X_{\alpha\vee_m\beta}^{(m)}$
or $X_\alpha^{(m)}X_\beta^{(m)}=0$. In the former case, $(\beta\vee_m\gamma,\alpha\vee_m\beta\vee_m\gamma)$, $(\gamma,\beta\vee_m\gamma)\in E_m$
implies that $(\gamma,\alpha\vee_m\beta\vee_m\gamma)\in E_m$ by the transitivity of $G_m$, so $\alpha\vee_m\beta\in \Xi(G_m)$.

For (iii), it is enough to notice that $(\alpha\vee_m\alpha\vee_m\gamma,\alpha\vee_m\gamma)=(\alpha\vee_m\gamma,\alpha\vee_m\gamma)\not\in E_m^\full$, because
$G_m^\full$ does not contain loop edges by definition.
\end{proof}

It is important to note that in general \emph{excitation operators corresponding to different reference states do not commute}: 
$X_{\alpha}^{(m)} X_{\beta}^{(\ell)}\neq X_{\beta}^{(m)} X_{\alpha}^{(\ell)}$ for $m\neq \ell$, again, because of \cref{differentvac}.

\begin{remark}\label{2qex}
The excitation operators are traditionally expressed using the language of second quantization.
Let $a_p^\dag$ and $a_p$ denote the fermionic creation and annihilation operators. Then $\Phi_\beta=a^\dag_{\beta_1}\cdots a^\dag_{\beta_N}\ket{\mathrm{vac}}$, where $\beta=\{\beta_1<\ldots<\beta_N\}$, and
$$
X_{\alpha}=a^\dag_{p_1} a_{q_1}\cdots a^\dag_{p_n} a_{q_n}.
$$
Here, $\ket{\mathrm{vac}}$ is the Fock vacuum state, $\{q_1,\ldots,q_n\}=0\setminus\occ{}{\alpha}$ and $\{p_1,\ldots,p_n\}=\virt{}{\alpha}$ with $q_1<\ldots<q_n$ and $p_1<\ldots<q_n$.
In other words, $X_{\alpha}$ changes the orbitals $0\setminus\occ{}{\alpha}$ to $\virt{}{\alpha}$, as expected.
Although the excitation operators commute with each other, they \emph{do not} commute in general with the Hamiltonian.
\end{remark}

We now define a family of operators which ``reverse'' the action of $X_{\alpha}^{(m)}$.
\begin{definition}
	Let $G_m=(L_m,E_m)$ be a subgraph of $G^\full_m$.
For all $\alpha\in\Xi(G_m)$, the linear operators $(X_{\alpha}^{(m)})^\dag:\HC^1_K\to \HC^1_K$ defined via 
$$
(X_{\alpha}^{(m)})^\dag \Phi_\beta=\begin{cases}
\sigma(\alpha,\alpha^\perp\wedge_m\beta)\Phi_{\alpha^\perp\wedge_m\beta} & (\alpha^\perp\wedge_m\beta,\beta)\in E_m \\
0 & (\alpha^\perp\wedge_m\beta,\beta)\not\in E_m
\end{cases}
$$
for any $\beta\in S$,
and extended boundedly and linearly to the whole space $\HC^1_K$, are called \emph{de-excitation operators on $G_m$}. 
\end{definition}

It is easy to see using \cref{abinvgraph} and \cref{Lmreform} that 
\begin{equation}\label{rkdeex}
\rk_m(\alpha^\perp\wedge_m\beta)=\rk_m(\beta)-\rk_m(\alpha),
\end{equation}
whenever $(\alpha^\perp\wedge_m\beta,\beta)\in E_m$.
Therefore, we may roughly say that the de-excitation operator $(X_{\alpha}^{(m)})^\dag$ decreases the rank by $\rk_m(\alpha)$.
Of course, the notation $\dag$ is not coincidental, and $(X_{\alpha}^{(m)})^\dag$ is in fact the $\LC^2$-adjoint of $X_{\alpha}^{(m)}$.

\begin{theorem}\label{deexadj}
Suppose that $\{X_\alpha^{(m)}\}$ and $\{(X_\alpha^{(m)})^\dag\}$ are the set of excitation and de-excitation
operators corresponding to the excitation graph $G_m$. Then 
$$
\dua{(X_\alpha^{(m)})^\dagger \Phi}{\Psi}=\dua{\Phi}{X_\alpha^{(m)}\Psi}\quad\text{for all}\quad \Phi,\Psi\in\HC^1_K\;\text{and}\;\alpha\in\Xi(G_m).
$$
\end{theorem}
\begin{proof}
It is enough to prove the relation for $\Phi=\Phi_\gamma$ and $\Psi=\Phi_\beta$, as the general statement follows by linearity.
Suppose that $(\alpha^\perp\wedge_m\gamma,\gamma)\in E_m$, then 
\begin{align*}
\dua{(X_\alpha^{(m)})^\dagger \Phi_\gamma}{\Phi_\beta}&=\sigma(\alpha,\alpha^\perp\wedge_m\gamma)\dua{\Phi_{\alpha^\perp\wedge_m \gamma}}{\Phi_\beta}=\sigma(\alpha,\beta)\dua{\Phi_\gamma}{\Phi_{\alpha\vee_m\beta}}=
\dua{\Phi_\gamma}{X_\alpha^{(m)}\Phi_\beta},
\end{align*}
where we used that $\alpha^\perp\wedge_m \gamma=\beta\in L_m$ if and only if $\alpha\vee_m\beta=\gamma\in L_m$ (\cref{abinvgraph}).
\end{proof}

\begin{theorem}\label{deexthm}
Let $G_m=(L_m,E_m)$ be a consistent subgraph of $G^\full_m$ and let $\{X_{\alpha}^{(m)}\}_{\alpha\in\Xi(G_m)}$ and $\{(X_{\alpha}^{(m)})^\dag\}_{\alpha\in\Xi(G_m)}$
 denote the set of excitation-, and de-excitation operators on $G_m$.
Then the following properties hold true.
\begin{enumerate}[label=\normalfont(\roman*)]
\item (commutativity) For all $\alpha,\beta\in \Xi(G_m)$, there holds 
$$
(X_{\alpha}^{(m)})^\dag (X_{\beta}^{(m)})^\dag=(X_{\beta}^{(m)})^\dag (X_{\alpha}^{(m)})^\dag.
$$
\item For any $\alpha,\beta\in \Xi(G_m)$ and $\gamma\in S$, the following formula holds true:
$$
(X_\alpha^{(m)})^\dagger X_\beta^{(m)} \Phi_\gamma=
\sigma(\alpha,\alpha^\perp\wedge_m(\beta\vee_m\gamma))\sigma(\beta,\gamma)\Phi_{\alpha^\perp\wedge(\beta\vee\gamma)}
$$
if $(\gamma, \beta\vee_m\gamma)\in E_m$ and $(\alpha^\perp\wedge_m(\beta\vee_m\gamma),\beta\vee_m\gamma)\in E_m$ both hold true.  Otherwise,
$(X_\alpha^{(m)})^\dagger X_\beta^{(m)} \Phi_\gamma=0$.
In particular, $(X_{\alpha}^{(m)})^\dag \Phi_\alpha=\Phi_{0_m}$.
\item $(X^{(m)}_\alpha)^\dag \Phi_{0_\ell}=0$ for any $m,\ell=1,\ldots,M$ and $\alpha\in\Xi(G_m)$.
\item (nilpotency) $((X^{(m)}_\alpha)^\dag)^2=0$ for any $\alpha\in\Xi(G_m)$.
\end{enumerate}
\end{theorem}
\begin{proof}
Part (i) follows from \cref{deexadj} combined with \cref{exopthm} (i). Part (ii) follows directly from the definitions.
Part (iii) comes from the fact that there are no edges between different $0_m$'s.
Part (iv) follows from  \cref{exopthm} (iii).
\end{proof}

It is highly important to stress that in general excitation-, and de-excitation operators do not commute with each other:
$$
X_{\alpha}^{(m)}(X_{\alpha}^{(m)})^\dag\neq (X_{\alpha}^{(m)})^\dag X_{\alpha}^{(m)},
$$
in other words, the $X_{\alpha}^{(m)}$'s are \emph{nonnormal} operators. Also, $[(X_{\alpha}^{(m)})^\dag, X_{\beta}^{(m)}]\neq 0$ in general.
 This fact is the source of many technical obstacles in the analysis
of the CC method, primarily because it implies that the similarity-transformed Hamilton operator \cref{bchham} is nonnormal.

\subsection{Cluster operators}\label{clustop}

From now on, we omit the reference index $m$ from the notations, with the understanding that the considerations hold true for every reference independently.
Suppose that we constructed the set of excitation operators $\{X_{\alpha}\}_{\alpha\in\Xi(G)}$ for a given consistent subgraph $G=(L,E)$.
The completion of their linear hull
$$
\VF(G)=\ol{\Span \{X_{\alpha}\}_{\alpha\in\Xi(G)}}^{\|\cdot\|_{\mathcal{L}(\HC^1,\HC^1)}}
$$
is called the \emph{space of cluster operators on $G$} endowed with operator norm $\|\cdot\|_{\mathcal{L}(\HC^1,\HC^1)}$. As mentioned earlier, if $G$ is not the full excitation graph $G^\full$, then certain excitation operators
will be absent and therefore, they will be missing from $\VF(G)$ as well. 

\begin{proposition}\label{clusnil}
For any $T\in\VF(G)$, we have $T^{N+1}=0$.
\end{proposition}
\begin{proof}
It is enough to prove that an arbitrary product of $N+1$ excitation operators is zero. In fact, by definition every excitation operator either increases the rank of a Slater determinant
by at least 1 or maps it to zero.
But the rank cannot increase above $N$, so the product must be zero.
\end{proof}

It is well-known that the vector space  $\VF(G^\full)$ constructed on the full excitation graph $G^{\full}$ 
forms a \emph{commutative algebra} (see e.g. \cite[Lemma 4.2]{schneider2009analysis}) with the usual multiplication (a subalgebra of the algebra of bounded linear operators $\mathcal{L}(\HC^1_K,\HC^1_K)$).
According to \cref{clusnil}, it is also \emph{nilpotent}. More generally, we have
\begin{theorem}\label{subalgebratrans}
$\VF(G)$ is a nilpotent, commutative algebra for any transitive excitation graph $G$.
\end{theorem}
\begin{proof}
Follows from \cref{exopthm} (ii).
\end{proof}

If, however, $G$ is not transitive, then $\VF(G)$ is \emph{not} an algebra in general---for instance in $\VF(G(\mathrm{SD}))$ there are no excitation operators of rank 3 and above, but the rank of the products of excitation operators can be arbitrary ($\le N$).

\begin{example}
We observed in \cref{tccgraph} that the CAS-subgraph $G(\mathrm{CAS})$ corresponding to the TCC method is transitive and consistent, hence
$\VF(G(\mathrm{CAS}))$ forms a subalgebra of $\VF(G^\full)$ (cf. \cite{kowalski2018properties}).
Similarly, for $G(\mathrm{int})$ in \cref{kowalskigraph}, $\VF(G(\mathrm{int}))$ also forms a subalgebra.
However, in a truncated setting, where only certain low-rank edges of $E(\mathrm{CAS})$ (or $E(\mathrm{int})$) are retained, transitivity, hence the subalgebra property
is lost.
\end{example}

Let now the excitation graph $G=(L,E)$ be arbitrary. A cluster operator $C\in\VF(G)$ may be decomposed according to the excitation ranks of its constituent excitations as
\begin{equation}\label{clusterrk}
C=\sum_{r=1}^N C_r, \quad\text{where}\quad C_r=\sum_{\rk(\alpha)=r} c_\alpha X_\alpha.
\end{equation}
We say that $C$ \emph{is of rank $r$} if it contains excitation operators of rank at most $r$.
Note that the graded structure of $G$ is compatible with this decomposition in the sense that if $C$ and $D$ are of ranks $r$ and $s$, respectively, then
$CD$ is of rank at most $r+s$. 

\begin{remark}
In the SR case, the cluster operators can be used to express any wavefunction in $\HC^1_K$ if the full excitation graph $G^\full$ is used for their construction.
In fact, in this case, $X_\alpha\Phi_0=\Phi_\alpha$ for \emph{every} $\alpha\in\ol{L}$, so
we may express any function in $\HC^1_K$ through a linear combination of the excitation operators \emph{and} the identity $I$. More precisely, if
$$
\Psi=\sum_{\alpha\in L} c_\alpha\Phi_\alpha= c_0\Phi_0 + \sum_{\alpha\in\ol{L}} c_\alpha\Phi_\alpha, \quad\text{then}\quad \Psi=\Bigg[c_0 I + \sum_{\alpha\in\ol{L}} c_\alpha X_\alpha\Bigg]\Phi_0,
$$
for some scalars $\{c_\alpha\}_{\alpha\in L}$.
Recall that in \cref{seccicc} we assumed the intermediate normalization condition $\dua{\Psi}{\Phi_0}=1$, which implies $c_0=1$.
There is a one-to-one correspondence between
functions $\Psi\in\HC^{1,\perp}_K$ and the cluster operators $C_\Psi$ defined as
\begin{equation}\label{clusterop}
C_\Psi=\sum_{\alpha\in\ol{L}} c_\alpha X_\alpha, \quad\text{where}\quad c_\alpha=\dua{\Psi}{\Phi_\alpha}.
\end{equation}
It is not clear, however, that $C_\Psi\in \mathcal{L}(\HC^1_K,\HC^1_K)$. See \cref{clusH1} below for the precise statement of this nontrivial fact.
Also, if the excitation graph does not contain every edge of the form $(0,\alpha)$---which is typically the case if some truncation is used---then
 it is \emph{not} possible to assign a cluster operator \cref{clusterop} to every $\Psi\in\HC^{1,\perp}_K$.
\end{remark}

The following important result makes the aforementioned correspondence between functions and cluster operators precise.
\begin{theorem}\cite[Theorem 4.1 and Lemma 5.1]{rohwedder2013continuous}\label{clusH1}
Fix $\Psi\in\HC^{1,\perp}$. Then, the following hold true.
\begin{enumerate}
\item The cluster operator $C_\Psi$ \cref{clusterop} satisfies $C_\Psi\in \mathcal{L}(\HC^1,\HC^1)$.
Furthermore, there is a constant $b>0$ independent of $\Psi$ such that
$$
\|\Psi\|_{\HC^1}\le \|C_\Psi\|_{\mathcal{L}(\HC^1,\HC^1)}\le b\|\Psi\|_{\HC^1}.
$$
\item $C_\Psi^\dag\in \mathcal{L}(\HC^1,\HC^1)$, and there is a constant $b'>0$ independent of $\Psi$ such that
$$
\|C_\Psi^\dag\|_{\mathcal{L}(\HC^1,\HC^1)}\le b'\|\Psi\|_{\HC^1},
$$
and there cannot be a uniform lower bound in terms of $\|\Psi\|_{\HC^1}$.
\item $C_\Psi$ can be extended to $\mathcal{L}(\HC^{-1},\HC^{-1})$.
\end{enumerate}
\end{theorem}

Next, we consider the so-called \emph{exponential Ansatz}, which is the representation
$$
I+C=e^{T}, \quad\text{where}\quad T=\sum_{\alpha\in\Xi(G^\full)} t_\alpha X_\alpha\in\VF(G^\full),
$$
and $C\in\VF(G^\full)$.
Here, $e^{T}$ is simply a finite sum due to the nilpotency of $T$, i.e.
$$
e^{T}=I+T+\tfrac{1}{2!}T^2+\ldots+\tfrac{1}{N!} T^N.
$$
The inverse of the exponential should be the logarithm, as one would expect.
\begin{theorem}\cite[Lemma 5.2]{rohwedder2013continuous}\label{explog}
For any cluster operator $C\in\VF(G^{\full})$ there exists a unique cluster operator $T\in\VF(G^{\full})$, such that $e^{T}=I+C$. Furthermore,
$$
T=\log(I+C)=C-\tfrac{1}{2}C^2+\tfrac{1}{3}C^3-\ldots+\tfrac{(-1)^{N-1}}{N}C^N.
$$
Moreover, the exponential map is a bijection between 
$$
\mathcal{S}=\Big\{ S \in \mathcal{L}(\HC^1,\HC^1) : S=\sum_{\alpha\in\ol{L}} s_\alpha X_\alpha \Big\}\quad\text{and}\quad I+\mathcal{S}.
$$
Furthermore, the result also holds true if $\mathcal{L}(\HC^1,\HC^1)$ is replaced with $\mathcal{L}(\HC^{-1},\HC^{-1})$.
\end{theorem}

It is important to note that if some proper excitation subgraph $G=(L,E)$ is considered instead of $G^{\full}$, the previous result does \emph{not} hold.
For instance, if $G(\mathrm{SD})$ is considered, then it might not be possible to represent $I+C$ as $e^T$, where $C\in\VF(G^\full)$ and $T\in\VF(G)$.
This in particular implies that wavefunctions of the form $e^{T}\Phi_{0}$ where $T\in\VF(G)$ is \emph{not} the totality
 of intermediately normalized wavefunctions.

In the multireference (MR) case, the analogue of the exponential Ansatz is called the \emph{Jeziorski--Monkhorst (JM) Ansatz}, see \cref{jmder} below. 
In the JM-MRCC method, $M$ wavefunctions,
say $\Psi_1,\ldots,\Psi_M$ are ``targeted'', and the expansion
\begin{equation}\label{jmansatz}
\Psi_j=\sum_{m=1}^M a_j^{(m)} e^{T^{(m)}}\Phi_{0_m}, \quad\text{where}\quad a_j^{(m)}\in\RR,
\end{equation}
is utilized. In the untruncated case, suppose that $\Psi_j=(I+C^{(j)})\Phi_{0_j}=e^{T^{(j)}}\Phi_{0_j}$, as above, for all $j=1,\ldots,M$.
Then the JM expansion coefficients $a_j^{(m)}$ of $\Psi_j$ are simply $\delta_{jm}$.

\subsection{Cluster amplitude spaces}\label{secamp}

The linear combination coefficients of the excitation operators making up a cluster operator are called \emph{cluster amplitudes}.
Let $\ell^2(G)$ denote Hilbert space of square summable real-, or complex-valued sequences indexed by the edge labels of the excitation graph $G$, i.e.
$$
\ell^2(G)=\{ t=(t_\alpha)_{\alpha\in \Xi(G)} : \|t\|_{\ell^2}<\infty\}.
$$
The (real or complex) Hilbert space 
$$
\mathbb{V}(G)=\{ t\in \ell^2(G) : \|T\Phi_{0}\|_{\HC^1}<\infty\},
$$
endowed with the $\HC^1$-inner product $\dua{t}{s}_\VV=\dua{T\Phi_0}{S\Phi_0}_{\HC^1}$
is called the \emph{(cluster) amplitude space corresponding to $G$}. Nevertheless, from now on we use the convention that the unmarked 
$\dua{t}{s}=\dua{T\Phi_0}{S\Phi_0}_{\LC^2}$ and $\|\cdot\|$ refers to the $\ell^2$-inner product and $\ell^2$-norm. 
Clearly, $\|t\|\le \|t\|_\VV$. 
\begin{remark}
Similarly to $\HC^1\hookrightarrow\LC^2\hookrightarrow\HC^{-1}$, the spaces $\VV(G)\hookrightarrow \ell^2(G)\hookrightarrow \VV(G)^*$
also form a Gelfand triple.
\end{remark}

It is clear that the space of cluster operators $\VF(G)$ is canonically isomorphic to $\VV(G)$ via
$$
\VF(G)\ni \sum_{\alpha\in\Xi(G)} c_\alpha X_\alpha =C\mapsto c=(c_\alpha)_{\alpha\in\Xi(G)}\in\VV(G).
$$
As customary in CC theory, we will never explicitly denote this isomorphism, and instead use capital letters $S,T,U,V,W$,~etc. to denote the cluster operators and small letters $s,t,u,v,w$,~etc. 
to denote their corresponding cluster amplitudes.

Furthermore, to every amplitude space $\VV(G)$ there corresponds a \emph{functional amplitude space} $\VC(G)\subset\HC^{1,\perp}$ 
through the $(\ell^2,\LC^2)$-isometric isomorphism $\VV(G)\to \VC(G)$ given by
$$
\VV(G)\ni c \mapsto C\Phi_0=\sum_{\alpha\in \Xi(G)} c_\alpha \Phi_\alpha\in \VC(G).
$$
Clearly, an appropriate subset of the Slater determinant basis $\mathfrak{B}_K$ (see \cref{detbas}) forms a basis of the functional amplitude space $\VC(G)$.

Given a closed subspace $\mathfrak{U}\subset \VC(G)$, we will sometimes use the orthogonal projector $\Pi_{\mathfrak{U}} : \LC^2\to\mathfrak{U}\subset\LC^2$ onto $\mathfrak{U}$,
defined as
$$
\dua{\Pi_{\mathfrak{U}} \Psi}{\Phi}=\dua{\Psi}{\Phi}, \quad \text{for all} \quad \Psi\in\LC^2, \; \Phi\in\mathfrak{U}.
$$
Hence, the inclusion map $I_{\mathfrak{U}} : \mathfrak{U} \to \LC^2$, given by $I_{\mathfrak{U}}\Phi=\Phi$ for all $\Phi\in\mathfrak{U}$ satisfies 
$I_{\mathfrak{U}}^\dag=\Pi_{\mathfrak{U}}$.

We continue by recalling an important notion due to \cite{schneider2009analysis}.
\begin{definition}
The excitation graph $G$ is said to be \emph{excitation complete}, if $\alpha^\perp\wedge\beta\in \Xi(G)$
for all $\alpha,\beta\in\Xi(G)$ with  $(\alpha^\perp\wedge\beta,\beta)\in E$ and $\alpha\neq\beta$.
\end{definition}

It is easy to see using \cref{rkdeex}, that commonly used rank-truncated graphs such as $G(1,2,\ldots,\rho)$ and $G(\mathrm{D})$ are excitation complete.

\begin{proposition}\cite[Lemma 5.5]{schneider2009analysis}\label{exexp}
Suppose that $G$ is excitation complete, let $\VC=\VC(G)$ and $\VC_0=\Span\{\Phi_0\}\oplus \VC$. Fix $t\in\VV$.
\begin{enumerate}[label=\normalfont(\roman*)]
\item The linear mappings
$e^{\pm T^\dag} I_{\VC_0} : \VC_0\to \VC_0$
are bijective.
\item The linear mappings 
$\Pi_{\VC} e^{\pm T^\dag} I_{\VC} : \VC\to \VC$
are surjective.
\end{enumerate}
\end{proposition}

The result follows easily from the next lemma.

\begin{lemma}\cite[Lemma 5.4]{schneider2009analysis}\label{excomplemm}
Suppose that $G$ is excitation complete. Then, for every $\alpha,\beta\in\Xi(G)$ we have $X_\alpha^\dag \Phi_\beta\in \VC(G)\cup \{\Phi_{0}\}$.
\end{lemma}
\begin{proof}
From \cref{deexthm} (ii), we have 
$$
X_\alpha^\dag \Phi_\beta = \sigma(\alpha,\alpha^\perp \wedge \beta) \Phi_{\alpha^\perp \wedge \beta},
$$
if $(\alpha^\perp \wedge \beta, \beta)\in E$. If $\alpha\neq\beta$, then right-hand side is in $\VC(G)$, since $G$ is excitation complete.
If $\alpha=\beta$, then the right-hand side is simply $\Phi_{0}$. 
\end{proof}

\begin{proof}[Proof of \cref{exexp}]
By linearity, \cref{excomplemm} implies that the mapping $T^\dag : \VC_0(G)\to \VC_0(G)$ and so $e^{\pm T^\dag} : \VC_0(G)\to \VC_0(G)$ as well.
But $(e^{T^\dag})^{-1}=e^{-T^\dag}$, which proves (i). Part (ii) follows easily from this.
\end{proof}

\section{Derivation of the Coupled-Cluster Equations}\label{ccdersec}
In this section, we give derivations of the SRCC-, and a variant of the MRCC equations. The approach presented here is based on \cite{wilson2013methods}.
 We would like to stress that the discussion only applies to the \emph{full} (that is, untruncated) CC methods.

The essence of the following theorem seems to be well-known in the physics and quantum chemistry literature, and the method itself
 is generally attributed to C. Bloch \cite{bloch1958theorie}, who devised it in the context of perturbation theory.
 
\begin{theorem}\label{blochthm}
Let $\mathfrak{H}$ and $\mathfrak{L}$ be (real or complex) Hilbert spaces so that they form a Gelfand triple: $\mathfrak{H}\subset\mathfrak{L}\subset\mathfrak{H}^*$.
Let $\ham:\mathfrak{H}\to\mathfrak{H}^*$ be a bounded operator.
Let $\mathfrak{M},\mathfrak{N}\subset \mathfrak{H}$ be any pair of closed subspaces so that the following \emph{complementarity condition} holds:
\begin{equation}\label{compcond}
\mathfrak{M} \oplus \mathfrak{N}^\perp=\mathfrak{H}.
\end{equation}
Then the following are equivalent.
\begin{enumerate}[label=\normalfont(\roman*)]
\item $\mathfrak{M}\subset \mathfrak{H}$ is \emph{weakly $\ham$-invariant:} for every $\Phi\in\mathfrak{M}$ there exists $\wt{\Phi}\in\mathfrak{M}$ such that
$\dua{\ham\Phi}{\Phi'}=\dua{\wt{\Phi}}{\Phi'}$ for all $\Phi'\in\mathfrak{H}$.
\item  (weak Bloch equation) There holds 
\begin{equation}\label{bloch}
\dua{\ham\Xi\Phi}{(I-\Xi^\dag)\Phi'}=0\quad\text{for all}\quad \Phi\in\mathfrak{N},\Phi'\in\mathfrak{N}^\perp,
\end{equation}
where $\Xi : \HC\to\HC$ denotes the (oblique) projector onto $\mathfrak{M}$ along $\mathfrak{N}^\perp$, i.e. 
$\Ran\Xi=\mathfrak{M}$ and  $\ker\Xi=\mathfrak{N}^\perp$.
\end{enumerate}
Furthermore, if
\begin{equation}\label{blochm}
\mathfrak{M}=\Span\{\Psi_j\in \mathfrak{H} : j=1,\ldots,J\}, \quad\text{where}\quad \dua{\ham\Psi_j}{\ol{\Phi}}=\ene_j\dua{\Psi_j}{\ol{\Phi}}\quad (\ol{\Phi}\in\mathfrak{H})
\end{equation}
for some $\ene_j\in\CC$, then with the \emph{effective Hamiltonian} $\ham^\eff:\mathfrak{N}\to\mathfrak{N}$, given by 
$\dua{\ham^\eff\Phi}{\Phi'}=\dua{\ham\Xi\Phi}{\Phi'}$ for all $\Phi,\Phi'\in\mathfrak{N}$, we have
\begin{equation}\label{effeig}
\dua{\ham^\eff\Pi\Psi_j}{\Phi}=\ene_j\dua{\Pi\Psi_j}{\Phi}\quad\text{for all}\quad \Phi\in\mathfrak{N},
\end{equation}
where $\Pi : \HC\to\HC$ denotes the $\LC$-\emph{orthogonal} projector onto $\mathfrak{N}$, i.e. $\Ran\Pi=\mathfrak{N}$ and $\ker\Pi=\mathfrak{N}^\perp$.
\end{theorem}

\begin{proof}
For (i)$\Longrightarrow$(ii), note that using $\ker\Xi=\mathfrak{N}^\perp$ and $\Ran\Xi=\mathfrak{M}$, it follows from (i) that for every
 $\Phi\in\mathfrak{N}$ there exists $\wt{\Phi}\in\mathfrak{M}$ such that $\dua{\ham\Xi\Phi}{\ol{\Phi}}=\dua{\wt{\Phi}}{\ol{\Phi}}$ for all $\ol{\Phi}\in\mathfrak{H}$. 
Put $\ol{\Phi}=(I-\Xi^\dag)\Phi'$ to obtain
$$
\dua{\ham\Xi\Phi}{(I-\Xi^\dag)\Phi'}=\dua{\wt{\Phi}}{(I-\Xi^\dag)\Phi'}=0 \quad\text{for all}\quad \Phi\in\mathfrak{N},\Phi'\in\mathfrak{H},
$$
where we used that $\wt{\Phi}\in\mathfrak{M}$ and $\Ran(I-\Xi^\dag)=\mathfrak{M}^\perp$.
From this, \cref{bloch} follows.

To see (ii)$\Longrightarrow$(i), fix $\Phi\in\mathfrak{M}$ and note that \cref{bloch} implies $F_\Phi(\Phi')=0$ for all $\Phi'\in\mathfrak{M}^\perp$,
where $F_\Phi(\Phi'):=\dua{\ham\Phi}{\Phi'}$ for all $\Phi'\in\mathfrak{H}$.
Here, $F_\Phi(\cdot)$ is a bounded linear functional on the dense subspace $\mathfrak{H}\subset\mathfrak{L}$. Extend $F_\Phi$ to a bounded linear functional $\wh{F}_\Phi$ on $\mathfrak{L}$. The Riesz representation theorem implies that there is a $\wt{\Phi}\in\mathfrak{L}$ such that $\wh{F}_\Phi(\Phi')=\dua{\wt{\Phi}}{\Phi'}$ for all $\Phi'\in\mathfrak{L}$.
But $0=F_\Phi(\Phi')=\wh{F}_\Phi(\Phi')=\dua{\wt{\Phi}}{\Phi'}$ for all $\Phi'\in\mathfrak{M}^\perp$, so $\wt{\Phi}\in\mathfrak{M}^{\perp\perp}=\mathfrak{M}$.
Therefore, we constructed a $\wt{\Phi}\in\mathfrak{M}$ such that $\dua{\ham\Phi}{\Phi'}=\dua{\wt{\Phi}}{\Phi'}$ for all $\Phi'\in\mathfrak{H}$, which is what we wanted to prove.

To prove the ``furthermore'' part, first note that $\mathfrak{M}$ is weakly $\ham$-invariant. We now claim that $\Xi\Pi=\Xi$. In fact, $\Ran(I-\Pi)=\ker\Pi=\ker\Xi$, so $\Xi(I-\Pi)=0$. 
Continuing the proof, note that the second relation of \cref{blochm} is equivalent to
$$
\dua{\ham\Xi\Pi\Psi_j}{\ol{\Phi}}=\ene_j\dua{\Xi\Pi\Psi_j}{\ol{\Phi}}\quad\text{for all}\quad \ol{\Phi}\in\mathfrak{H}.
$$
Using \cref{bloch}, this can be further written as
$$
\dua{\ham\Xi\Pi\Psi_j}{\Xi^\dag\ol{\Phi}}=\ene_j\dua{\Pi\Psi_j}{\Xi^\dag\ol{\Phi}}\quad\text{for all}\quad \ol{\Phi}\in\mathfrak{H}.
$$
The desired result follows by noting that $\Ran\Xi^\dag=\mathfrak{N}$.
\end{proof}

In practice, $\mathfrak{M}$ (called the ``exact model space'') is unknown and $\mathfrak{N}$ (called the ``model space'')
is chosen in a way that it provides a ``reasonable approximation'' to $\mathfrak{M}$, i.e. that \cref{compcond}
holds. In particular, $\mathfrak{M}\subset\mathfrak{N}^\perp$ is not permitted. 
Then, the unknown ``wave operator'' $\Xi$ (hence $\mathfrak{M}$) can be determined by solving the weak Bloch equation \cref{bloch}. 
Next, the eigenvalue problem for $\ham^\eff$
is solved to obtain the energies $\ene_1,\ldots,\ene_M$ and (some of the) eigenvectors. 
 
\begin{remark}\leavevmode
\begin{enumerate}[label=\normalfont(\roman*)]
\item It is important to note that solving the Bloch equation only provides a weakly $\ham$-invariant subspace $\mathfrak{M}$ and it might \emph{not}
be a direct sum of (weak) eigenspaces in general. In other words, $\mathfrak{M}$ might be spanned by an incomplete set of eigenvectors.
Clearly, in such a situation some of the eigenvectors cannot be recovered through solving the eigenproblem for the effective Hamiltonian $\ham^\eff$.
\item The Bloch equation \cref{bloch} is more commonly given in the ``strong'' form ``$\Xi\ham\Xi=\ham\Xi$''.
\end{enumerate}
\end{remark}

The situation is greatly simplified, when one considers one-dimensional subspaces $\mathfrak{N}$ and $\mathfrak{M}$,
because a one-dimensional invariant subspace is always an eigenspace.

\begin{corollary}\label{blochcoro}
Let $\dim\mathfrak{N}=\dim\mathfrak{M}=1$, and set $\mathfrak{N}=\Span\{\Phi_0\}$ for some $\Phi_0\in\mathfrak{H}$.
Further, let $\mathfrak{M}=\Span\{\Psi\}$ for some $\Psi\in\mathfrak{H}$, and suppose that $\dua{\Psi}{\Phi_0}=1$.
Then, the following are equivalent.
\begin{enumerate}[label=\normalfont(\roman*)]
\item  $\dua{\ham\Psi}{\ol{\Phi}}=\ene\dua{\Psi}{\ol{\Phi}}$ for all $\ol{\Phi}\in\mathfrak{H}$ and some scalar $\ene$.
\item  $\dua{\ham\Xi\Phi_0}{(I-\Xi^\dag)\Phi'}=0$ for all $\Phi'\in\mathfrak{N}^\perp$.
\end{enumerate}
Furthermore, $\ene=\dua{\ham\Xi\Phi_0}{\Phi_0}$.
\end{corollary}

\subsection{The SRCC method}\label{srccdersec}

The single-reference Coupled-Cluster method easily follows from \cref{blochcoro} through the exponential parametrization of the wave operator.
In the following theorem, we re-establish \cite[Theorem 5.3]{rohwedder2013continuous} (see \cref{fccfci}).

\begin{theorem}\label{blochsrcc}
Let $\ham:\HC^1_K\to(\HC_K^1)^*$ be a bounded operator.
Fix $\Phi_0\in\HC^1_K$ with $\|\Phi_0\|=1$ and suppose that $\Psi\in\HC^1_K$ is such that  $\dua{\Psi}{\Phi_0}=1$.
Then the following are equivalent. 
\begin{enumerate}[label=\normalfont(\roman*)]
\item $\dua{\ham\Psi}{\Phi}=\ene \dua{\Psi}{\Phi}$ for all $\Phi\in\HC^{1}_K$ for some scalar $\ene$.
\item (Full CC) $\Psi=e^{T_*}\Phi_0$ for some $t_*\in\VV(G^\full)$ such that 
\begin{equation}\label{srccdereq}
\dua{e^{-T_*}\ham e^{T_*} \Phi_0}{S\Phi_0}=0\quad\text{for all}\quad s\in\VV(G^\full).
\end{equation}
Furthermore, $\ene=\dua{e^{-T_*}\ham e^{T_*} \Phi_0}{\Phi_0}$.
\item (Full CI) $\Psi=(I+C_*)\Phi_0$ for some $c_*\in\VV(G^\full)$ such that 
\begin{equation}\label{srcidereq}
\dua{\ham (I+C_*)\Phi_0}{S\Phi_0}=\ene_{\ci} \dua{(I+C_*)\Phi_0}{S\Phi_0}\quad\text{for all}\quad s\in\VV(G^\full),
\end{equation}
where $\ene_\ci=\dua{\ham (I+C_*)\Phi_0}{\Phi_0}$. Furthermore, $\ene=\ene_\ci$.
\end{enumerate}
\end{theorem}
\begin{proof}
Let $\mathfrak{H}=\HC^1_K$ and $\mathfrak{L}=\LC^2$. First, we prove (i)$\Longleftrightarrow$(ii).
We apply \cref{blochcoro} with the SRCC wave operator 
$$
\Xi=e^{T_*}\Pi_{\Phi_0},
$$
where $T_*$ is some cluster operator and $\Pi_{\Phi_0}$ is the orthogonal projector onto $\mathfrak{N}=\Span\{\Phi_0\}$. Note that $\mathfrak{N}^\perp=\VC(G^\full)$.  It is easy to see that $\Xi$ is idempotent, and that $\ker\Xi=\mathfrak{N}^\perp$.
By an appropriate choice of $T_*$, $\Ran\Xi=\mathfrak{M}$ using $\dua{\Psi}{\Phi_0}=1$ and \cref{explog}.
Furthermore, $\Span\{e^{T_*}\Phi_0\}=\Ran\Xi\subset\mathfrak{H}$ due to \cref{clusH1}. 
Applying \cref{blochcoro}, (i) holds if and only if $\Psi=e^{T_*}\Phi_0$ and $T_*$ satisfies the weak Bloch equation
$$
\dua{\ham e^{T_*} \Phi_0}{(I-\Pi_{\Phi_0}e^{T_*^\dag})S'\Phi_0}=0\quad\text{for all}\quad s'\in\VV(G^\full) .
$$
Recalling \cref{exexp} (ii), and using the change of variables $S'=e^{-T_*^\dag}S$,
$$
\dua{e^{-T_*}\ham e^{T_*} \Phi_0}{S\Phi_0}=0\quad\text{for all}\quad s\in\VV(G^\full).
$$
Here we used that $e^{-T_*}$ can be extended to a bounded $\HC^{-1}\to\HC^{-1}$ operator (\cref{clusH1}).\footnote{We refer the reader to the proof of \cite[Theorem 5.3]{rohwedder2013continuous} for more details.}
Note that $\ham^\eff$ is now a one-dimensional linear map (i.e. a multiplication by a scalar), so 
$\sigma(\ham^\eff)=\dua{e^{-T_*}\ham e^{T_*} \Phi_0}{\Phi_0}=\ene$.

Next, we prove (i)$\Longleftrightarrow$(iii). We now apply \cref{blochcoro} with the SRCI wave operator 
$$
\Xi=(I+C_*)\Pi_{\Phi_0},
$$
where $C_*$ is some cluster operator and the claim follows from a straightforward calculation. Further, 
now $\sigma(\ham^\eff)=\dua{\ham(I+C_*)\Phi_0}{\Phi_0}=\ene$.
\end{proof}

\subsection{The Jeziorski--Monkhorst MRCC method}\label{jmder}

In MRCC methods the ``model space'' $\mathfrak{N}$ is chosen to be the space spanned by $M$ orthonormal reference determinants,
$$
\mathfrak{N}=\Span\{ \Phi_{0_m} : m=1,\ldots,M\}.
$$ 
The \emph{Jeziorski--Monkhorst method} \cite{jeziorski1981coupled} uses the following Ansatz for the wave operator:
\begin{equation}\label{jmwave}
\Xi=\sum_{m=1}^M e^{T^{(m)}} \Pi_{\Phi_{0_m}},
\end{equation}
which corresponds to \cref{jmansatz}.

\begin{theorem}
Let $\mathfrak{N}$ be defined as above and set $\mathfrak{M}=\Span\{\Psi_m : m=1,\ldots,M\}$, where $\{\Psi_m\}_{m=1}^M\subset\HC^1_K$ is $\LC^2$-orthogonal.
Suppose that for every $m=1,\ldots,M$, $\dua{\Psi_m}{\Phi_{0_n}}\neq 0$ for at least one $n=1,\ldots,M$. Then, the following are equivalent.
\begin{enumerate}[label=\normalfont(\roman*)]
\item $\mathfrak{M}$ is weakly $\ham$-invariant: for every $\Psi_m$ ($m=1,\ldots,M$) there exists $\wt{\Psi}_m\in\mathfrak{M}$ such that
$\dua{\ham\Psi_m}{\Phi'}=\dua{\wt{\Psi}_m}{\Phi'}$ for all $\Phi'\in\mathfrak{H}^1_K$.
\item (Full JM-MRCC) $\mathfrak{M}=\Span\{ e^{T_*^{(m)}} \Phi_{0_m} : m=1,\ldots,M\}$, where $t_*^{(m)}\in \VV(G_m^\full)$ satisfies 
\begin{equation}\label{jmmrcc}
\dua{e^{-T_*^{(m)}}\ham e^{T_*^{(m)}} \Phi_{0_{m}}}{S^{(m)}\Phi_{0_m}}=\sum_{n=1}^M \ham^\eff_{mn}\dua{e^{-T_*^{(m)}} e^{T_*^{(n)}}\Phi_{0_n} }{S^{(m)}\Phi_{0_m}},
\end{equation}
for all $s^{(m)}\in \VV(G_m^\full)$ and $m=1,\ldots,M$, where the matrix elements of the effective Hamiltonian are given by
$\ham_{mn}^\eff=\dua{ e^{-T_*^{(m)}}\ham e^{T_*^{(m)}}\Phi_{0_{m}} }{\Phi_{0_n}}$. 
\item (Full MRCI) $\mathfrak{M}=\Span\{ (I+C_*^{(m)}) \Phi_{0_m} : m=1,\ldots,M\}$, where $c_*^{(m)}\in \VV(G_m^\full)$ satisfies 
\begin{equation}\label{jmmrci}
\dua{\ham (I+C_*^{(m)}) \Phi_{0_{m}}}{S^{(m)}\Phi_{0_m}}=\sum_{n=1}^M \wh{\ham}^\eff_{mn}\dua{(I+C_*^{(n)})\Phi_{0_n} }{S^{(m)}\Phi_{0_m}},
\end{equation}
for all $s^{(m)}\in \VV(G_m^\full)$ and $m=1,\ldots,M$, where the matrix elements of the effective Hamiltonian are given by
$\wh{\ham}_{mn}^\eff=\dua{ \ham (I+C_*^{(m)}) \Phi_{0_{m}} }{\Phi_{0_n}}$.
\end{enumerate}
Furthermore, suppose that $\dua{\ham\Psi_m}{\ol{\Phi}}=\ene_m\dua{\Psi_m}{\ol{\Phi}}$ for all $\ol{\Phi}\in\HC^1_K$ and $m=1,\ldots,M$.
Then the following hold true.
\begin{enumerate}[label=\normalfont(\alph*)]
\item Suppose $\mathfrak{M}$ is given as in (ii). Then the coefficients $a^{(m)}_j$ in the expansion
$\Psi_j=\sum_{n=1}^M a^{(n)}_j e^{T_*^{(n)}}\Phi_{0_n}$ are given as the solution to the eigenvalue problem
$$
\sum_{n=1}^M \ham^\eff_{nm} a_j^{(n)}=\ene_j a_j^{(m)}\quad\text{where}\quad m=1,\ldots,M.
$$
\item Suppose $\mathfrak{M}$ is given as in (iii). Then the coefficients $\wh{a}^{(m)}_j$ in the expansion
$\Psi_j=\sum_{n=1}^M \wh{a}^{(n)}_j (I+C_*^{(n)})\Phi_{0_n}$ are given as the solution to the eigenvalue problem
$$
\sum_{n=1}^M \wh{\ham}^\eff_{nm} \wh{a}_j^{(n)}=\ene_j \wh{a}_j^{(m)}\quad\text{where}\quad m=1,\ldots,M.
$$
\end{enumerate} 
\end{theorem}
\begin{proof}
Let $\mathfrak{H}=\HC^1_K$. First, we prove (i)$\Longleftrightarrow$(ii) by applying \cref{blochsrcc}.
Clearly, for the JM wave operator \cref{jmwave} we have $\Xi^2=\Xi$ and $\ker\Xi=\mathfrak{N}^\perp$ and
$$\Ran\Xi=\Span\{ e^{T^{(m)}} \Phi_{0_m} : m=1,\ldots,M\} .$$
The weak Bloch equation \cref{bloch} is equivalent to 
\begin{align*}
\dua{ \ham e^{T_*^{(m)}}\Phi_{0_{m}} }{\Phi'}=\sum_{n=1}^M \dua{ \ham e^{T_*^{(m)}} \Phi_{0_{m}} }{ \Pi_{\Phi_{0_n}} e^{(T_*^{(n)})^\dag} \Phi' }
\end{align*}
for all $\Phi'\in \mathfrak{N}^\perp$ and $m=1,\ldots,M$. Setting $\Phi'=S^{(m)}\Phi_{0_m}$, 
we obtain 
\begin{align*}
\dua{ \ham e^{T_*^{(m)}}\Phi_{0_{m}} }{S^{(m)}\Phi_{0_m}}&=\sum_{n=1}^M \dua{ \ham e^{T_*^{(m)}} \Phi_{0_{m}} }{ \Pi_{\Phi_{0_n}} e^{(T_*^{(n)})^\dag} S^{(m)}\Phi_{0_m}}\\
&=\sum_{n=1}^M \dua{ \ham e^{T_*^{(m)}} \Phi_{0_{m}} }{\Phi_{0_n}}\dua{e^{(T_*^{(n)})^\dag} S^{(m)}\Phi_{0_m}}{\Phi_{0_n}}\\
&=\sum_{n=1}^M \dua{ e^{-T_*^{(m)}}\ham e^{T_*^{(m)}} \Phi_{0_{m}} }{\Phi_{0_n}}\dua{e^{T_*^{(n)}}\Phi_{0_n}}{S^{(m)}\Phi_{0_m}}
\end{align*}
for all $s^{(m)}\in\VV(G^\full_m)$. Here, we used that $(T^{(m)})^\dag\Phi_{0_n}=0$, see \cref{deexthm} (iii).
The proof of \cref{jmmrcc} is finished by invoking \cref{exexp} (ii) and replacing $S^{(m)}$ by  $(e^{-T_*^{(m)}})^\dag S^{(m)}$.

Next, we prove (i)$\Longleftrightarrow$(iii). The MRCI wave operator reads
$$
\Xi=\sum_{m=1}^M (I+C_*^{(m)}) \Pi_{\Phi_{0_m}}.
$$
With this choice \cref{bloch} is equivalent to 
\begin{align*}
\dua{ \ham (I+C_*^{(m)})\Phi_{0_{m}} }{\Phi'}=\sum_{n=1}^M \dua{ \ham (I+C_*^{(m)}) \Phi_{0_{m}} }{ \Pi_{\Phi_{0_n}} (I+C_*^{(n)})^\dag \Phi' }
\end{align*}
for all $\Phi'\in \mathfrak{N}^\perp$ and $m=1,\ldots,M$.  Setting $\Phi'=S^{(m)}\Phi_{0_m}$, this can be written as 
\begin{align*}
\dua{ \ham (I+C_*^{(m)})\Phi_{0_{m}} }{S^{(m)}\Phi_{0_m}}&=\sum_{n=1}^M \dua{ \ham (I+C_*^{(m)}) \Phi_{0_{m}} }{ \Pi_{\Phi_{0_n}} (I+C_*^{(n)})^\dag S^{(m)}\Phi_{0_m} }\\
&=\sum_{n=1}^M \dua{ \ham (I+C_*^{(m)}) \Phi_{0_{m}} }{\Phi_{0_n}} \dua{(I+C_*^{(n)})^\dag S^{(m)}\Phi_{0_m} }{\Phi_{0_n}}\\
&=\sum_{n=1}^M \dua{ \ham (I+C_*^{(m)}) \Phi_{0_{m}} }{\Phi_{0_n}} \dua{(I+C_*^{(n)})\Phi_{0_n}}{S^{(m)}\Phi_{0_m}},
\end{align*}
which is what we wanted to prove.

For the ``furthermore'' part of (a), expanding $\Psi_j$ as 
$\Psi_j=\sum_{n=1}^M a^{(n)}_j e^{T_*^{(n)}}\Phi_{0_n}$,
for some scalars $a^{(n)}_j$, we find that $a^{(m)}_j=\dua{\Psi_j}{\Phi_{0_m}}$.
It is easy to see that \cref{effeig} now reads
$$
\sum_{n=1}^M \dua{\ham e^{T_*^{(n)}} \Phi_{0_n}}{\Phi_{0_m}} a_j^{(n)}=\ene_j a_j^{(m)}
$$
for all $j=1,\ldots,M$. The proof of the ``furthermore'' part of (b) is similar.
\end{proof}


\section{Conclusions and further work}

In this first part of a series of two articles, we proposed a framework to describe the discretization scheme involved in CC-like methods.
At the core of the description is the concept of the excitation graph (\cref{exgraphdef}), which completely determines all necessary building blocks such as 
excitation operators (\cref{exop}), cluster operators (\cref{clustop}) and cluster amplitude spaces (\cref{secamp}). 
The excitation graph concept admits a straightforward extension to the multireference case (\cref{exmultidef}).
Another advantage of our approach is that it avoids the use of second-quantized formalism and hence allowed us to prove the basic results 
(such as \cref{exopthm} and \cref{deexthm}) in a more transparent manner. Besides these, 
we also pointed out a number of structural properties of the excitation graph in \cref{exgrsec}.
 It is important to note that some of these graph-theoretic properties are reflected in the algebraic structure of the excitation operators 
 (\cref{consex} and \cref{subalgebratrans}).
Some relevant combinatorial quantities have been calculated in \cref{appgraph}. Furthermore, we proposed an algorithm to determine the reference states 
in an optimal fashion for the multirefence case in \cref{optimalmulti}. 

In \cref{ccdersec}, we provided unified and rigorous derivations of both the single-reference- (\cref{srccdersec}), and a multireference (\cref{jmder}) CC method. 
The derivations used a general theorem (\cref{blochthm}) motivated by a known method based on perturbation theory.


\appendix

\section{Properties of the excitation graph}\label{appgraph}

Here, we restrict ourselves to the single-reference case ($M=1$) and drop the subscript $m$'s from the notation.
Recall that $K$ denotes the cardinality of the orbital set $\Lambda$.
Given $\gamma\in L$, we introduce the set of paths of length $n$ from $0$ to $\gamma$ in $G$,
$$
\PP^{n}(\gamma)=\{\mi{\alpha}\in L\times\ldots\times L : \text{there is a path $0\to\gamma$ in $G$ having edges $\mi{\alpha}$} \}.
$$
The following theorem sheds light on the combinatorial structure of the excitation graph.
\begin{theorem}\label{graphthm}
Let $G^\full=(L,E^\full)$ be the full SR excitation graph with $K$ orbitals and $2N\le K$ particles.
Then the following properties hold true.
\begin{enumerate}[label=(\roman*)]
\item The number of vertices in $G$ is given by $|L|={K \choose N}$.
\item The number of vertices of rank $r$ is $|L(r)|={N \choose r }{K-N \choose r}$.
\item There are no edges in $E^\full$ entirely inside $L(r)$, and the number of edges from $L(r)$ to $L(r+s)$ is given by
$$
|E(r,r+s)|={K-N \choose r}{ K-N-r \choose s }{N \choose s+r}{s+r \choose r},
$$
for all $r=0,1,\ldots,N$ and $s=0,\ldots,N-r$, and $|E(r,r+s)|=0$ if $s=N-r+1,\ldots,N$. Furthermore, the symmetry property
$|E(r,r+s)|=|E(s,r+s)|$ holds true.
\item The total number of edges is given by
$$
|E^\full|=\sum_{r=1}^{N} {N \choose r }{K-N \choose r}{ K-2r \choose N-r }.
$$
\item The number of directed paths of length $n\le r=\rk(\gamma)$ from $0$ to $\gamma$ is given by $|\PP^{n}(\gamma)|=p(r,n)$, where
\begin{equation}\label{numpaths}
p(r,n)=\sum_{\substack{r_1+\ldots+r_n=r\\r_1,\ldots,r_n\ge 1}} \left(\frac{r!}{r_1!\cdots r_n!}\right)^2.
\end{equation}
\end{enumerate}
\end{theorem}
\begin{proof}
(i) is trivial, so is (ii). As for (iii),
we enumerate the pairs $(\alpha,\beta)$ in $E^\full$ as follows. Fix $\alpha$ with $\rk(\alpha)=r$, then
$\beta$ must satisfy $r+s\le N$, where $\rk(\beta)=s$, so that $|\ul{\alpha}\cup\ul{\beta}|=N$ is possible.
In $\ul{\beta}$, we must choose the missing internal letters from $\ul{\alpha}$ and there are $r$ of them. For the remaining $N-s-r$ elements, we may choose freely:
there are ${ N-r \choose N-s-r }$ possibilities to do this. Next, $\ol{\beta}$ must be disjoint from $\ol{\alpha}$, so there are $M-N-r$ letters
to choose from, giving ${ M-N-r \choose s }$ possibilities. Multiplying these independent choices by the number of ways $\alpha$ can be chosen for fixed $r$, we get
\begin{equation}\label{rsedge}
{N \choose r }{M-N \choose r} { N-r \choose N-s-r }{ M-N-r \choose s }
\end{equation}
for $s=1,\ldots,N-r$.
This can be rewritten using the formula ${n\choose h}{n-h \choose k}={n \choose k}{n-k \choose h}$ as
\[
{M-N \choose r}{ M-N-r \choose s }{N \choose s+r}{s+r \choose r}.
\]
Using the aforementioned formula for the first two factors, we also get the desired symmetry property.

Next, to derive (iv) we sum up \cref{rsedge},
$$
|E^\full|=\sum_{r=0}^{N}\sum_{s=1}^{N-r} {N \choose r }{M-N \choose r} { N-r \choose N-s-r }{ M-N-r \choose s }.
$$
Using Vandermonde's identity,
$$
\sum_{s=1}^{N-r}{ N-r \choose N-s-r }{ M-N-r \choose s }={ M-2r \choose N-r } - 1,
$$
we get
$$
|E^\full|=\sum_{r=1}^{N} {N \choose r }{M-N \choose r}{ M-2r \choose N-r },
$$
where we used Vandermonde's identity once more.

Next, we prove (v). We need to change $0$ into $\gamma$ in $n$ steps (edges). Suppose that the rank-increment of each step is $r_1,\ldots,r_n$, and are such that $r_1+\ldots+r_n=r$. 
In the $k$th step we replace letters $(\alpha_1,\ldots,\alpha_{r_k})$ with $(\beta_1,\ldots,\beta_{r_k})$. These choices can be done independently, so there are
$r!^2$ possibilities. However, the order of the $\alpha$'s and $\beta$'s is irrelevant in each step so we have to divide by $(r_1!\cdots r_n!)^2$. Summing over all $r_1,\ldots,r_n$
gives the stated formula.
\end{proof}

\begin{remark}\hspace*{\fill}
\begin{enumerate}[label=(\roman*)]
\item It follows that the vertex density per rank is hypergeometric,
\begin{equation}\label{vertdens}
\nu_r=\frac{{N \choose r} {K-N \choose M-N-r}}{{K \choose N}}, \quad\text{where}\quad r=0,1,\ldots,N.
\end{equation}
Therefore, its mean is $\frac{N}{K}(K-N)$ and its variance is $\frac{(K-N)^2 N^2}{(K-1)K^2}$.
\item The formula \cref{numpaths} implies that $|\PP^{n}(\gamma)|$ is independent of $N$ and $M$ and is constant for all $\gamma$ of fixed rank $r$.
\item If S truncation is in effect, we have $p_{\mathrm{S}}(r,n)=r!^2$ if $r=n$ and 0 otherwise. 
\item For the SD truncation, note that the number of $(r_1,\ldots,r_n)$ tuples with $r_j\in\{1,2\}$,
 $r_1+\ldots+r_n=r$ and $|\{j : r_j=2\}|=k$ is given by ${n \choose k}$ if $r=n+k$ and $0$ otherwise. Therefore,
$$
p_{\mathrm{SD}}(r,n)=\frac{r!^2}{4^{r-n}} { n \choose r-n }.
$$
\item According to the proof of \cite[Lemma 4.4.]{rohwedder2013continuous}, 
$$
|\{\beta\in L : \beta\preceq\alpha\}|=\sum_{s=1}^{r-1} {r\choose s} {r-1 \choose r-s},
$$
where $r=\rk(\alpha)$.
\end{enumerate}
\end{remark}

\section{Optimal choice of multireference determinants}\label{optimalmulti}

In this appendix, we describe an algorithm that can be used to automatically determine an optimal set of multireference determinants. 
Let $J\in\NN$ and let 
$$
\{\gamma_1,\ldots,\gamma_J\}\subset S
$$ 
be a fixed set of determinants. Also, fix an excitation rank truncation, e.g. S, SD, SDT, etc.
We want to select a \emph{minimal} set of reference elements $\Omega=\{0_1,\ldots,0_M\}$, so that each $\gamma_j$ is reachable through a \emph{direct} S, SD, SDT, etc. excitation
from $\Omega$, this is called ``first-order interaction space'' in MRCC theory.

Recall that each $\alpha\in 2^\Lambda$ can be represented as a binary characteristic vector $\bve{\alpha}\in \{0,1\}^K$ such that 
$$
\bve{\alpha}^t=\begin{cases}
1 & t\in \alpha\\
0 & t\not\in \alpha
\end{cases}
$$
The set $\{0,1\}^K$ endowed with the Hamming metric
$$
d_{\HA}(\bve{\alpha},\bve{\beta})=|\{t : \bve{\alpha}^t\neq \bve{\beta}^t, \; t=1,\ldots,K \}|
$$
is a complete metric space, called the Hamming space. The closed balls and the spheres in this space are denoted as $B_\HA(\bve{\alpha},R)$ and $S_\HA(\bve{\alpha},R)$.
Using this language, $S$ is simply $S_\HA(\bve{0},N)$, where $\bve{0}=(0,\ldots,0)$.\footnote{We warn the reader that the notation $\bve{0_m}$ for the vector representation of $0_m$ is slightly colliding with $\bve{0}$, the actual zero vector for the Hamming space.} Further, 
$$
\rk_m(\alpha)=\frac{1}{2}d_\HA(\bve{0_m},\bve{\alpha})
$$
for any $m=1,\ldots,M$.
Notice that $d_\HA(\bve{\alpha},\bve{\beta})\ge 2$ for distinct $\bve{\alpha},\bve{\beta}\in S_\HA(\bve{0},N)$.

This way, our optimization problem may be formulated as a covering problem in Hamming space. Let $\rho$ denote the excitation rank truncation, e.g.
$\rho=1,2,3,\ldots$ for S, SD, SDT, etc. Fix $J\in\NN$ and $\Gamma=\{\bve{\gamma}_1,\ldots,\bve{\gamma}_J\}\subset S_\HA(\bve{0},N)$. 
We need to find a minimal set of Hamming balls $\{B_\HA(\bve{0_m},2\rho) : m=1,\ldots,M\}$ with $\bve{0_m}\in S_\HA(\bve{0},N)$ such that 
$$
\Gamma\subset \bigcup_{m=1}^M B_\HA(\bve{0_m},2\rho)\cap S_\HA(\bve{0},N).
$$
Obviously, $\bve{0_m}\in \Gamma_{2\rho}$, where
$$
\Gamma_{2\rho}=\bigcup_{j=1}^J B_\HA(\bve{\gamma_j},2\rho)\cap S_\HA(\bve{0},N).
$$
In other words, it is sufficient to look for the $\bve{0_m}$'s in the much smaller set $\Gamma_{2\rho}$.
Let $n=|\Gamma_{2\rho}|$, and introduce some indexing in $\Gamma_{2\rho}$, say $\Gamma_{2\rho}=\{\bve{\alpha}_1,\ldots,\bve{\alpha}_n\}$. The geometric form of the covering problem may be rephrased as a binary integer linear
program (BILP) \cite{schrijver1998theory},
$$
\left.
\begin{aligned}
\sum_{\nu=1}^n \vec{c}_\nu\vec{x}_\nu&\to \min!\\
\sum_{\substack{\bve{\gamma}_j\in B_\HA(\bve{\alpha}_\nu,2\rho)\\1\le \nu\le n}}  \vec{x}_\nu&\ge 1,\quad j=1,\ldots,J\\
\vec{x}&\in \{0,1\}^n
\end{aligned}
\right\}
$$
where $\vec{c}\in\mathbb{Q}^n$ is a given rational cost vector.

\begin{remark}
If $\vec{c}_\nu=0$ for some $\nu$, then we will automatically have $\vec{x}_\nu=1$ in the solution,
even if  $B_\HA(\bve{\alpha}_\nu,2\rho)$ does not cover. On the other hand,
assigning a larger (resp. infinite) cost $\vec{c}_\nu$ will likely (resp. surely) end up $\vec{x}_\nu=0$ in the solution.
\end{remark}

The above problem is called a ``multidimensional knapsack problem'' in the optimization community, which seems to be extensively studied.
However, we just naively solve the BILP using general ILP methods available in \emph{Mathematica}.
In our experience, the BILP can be built up and solved in a small amount of time for practically relevant parameters $N$, $K$ and $J$, even on an older machine. 

The reason for the apparent efficiency might be that the number of variables $n$ in the the BILP above is significantly less then $|S|={K\choose N}$. In fact,
using the binary entropy function $H(x)=-x\log_2x-(1-x)\log_2(1-x)$, we have the rough estimate
$$
\frac{n}{|S|}\le \frac{J|B_\HA(\bve{0},2\rho)|}{|B_\HA(\bve{0},N)|}\le J \sqrt{8K\lambda'(1-\lambda')} 2^{-K (H(\lambda')-H(\lambda))},
$$
where $\lambda=2\rho/K$ and $\lambda'=N/K$ valid for $0<\lambda,\lambda'<\frac{1}{2}$ \cite[Lemma 2.4.4]{cohen1997covering}. Notice that $H(\lambda')\ge H(\lambda)$,
so $n/|S|\to 0$ as $K\to\infty$.


\vspace{1em}

\noindent\emph{Acknowledgements.}
The authors would like to thank Fabian M. Faulstich and Simen Kvaal for helpful discussions and comments on the manuscript. 
The useful suggestions of the anonymous reviewer are gratefully acknowledged. 
This work has received funding from the Norwegian Research Council through Grant Nos. 287906 (CCerror) and 262695 (CoE Hylleraas Center for Quantum Molecular Sciences).

\bibliographystyle{abbrv}
\bibliography{cc}

\end{document}